\documentclass[10pt,letterpaper]{amsart}
\usepackage{latexsym, amscd, amsfonts, eucal, mathrsfs, amsmath, amssymb, amsthm, tikz-cd, xr, stmaryrd, amsxtra, MnSymbol,bbm,mathrsfs,mathtools}
\usepackage[alphabetic]{amsrefs}
\usepackage{enumitem}
\usepackage[hidelinks]{hyperref}
\usepackage{chngcntr}

\counterwithin{equation}{section}
\newtheorem{theorem}[equation]{Theorem}
\newtheorem{lemma}[equation]{Lemma}

\newtheorem{proposition}[equation]{Proposition}
\newtheorem{corollary}[equation]{Corollary}

\newtheorem*{claim}{Claim}

\theoremstyle{definition}
\newtheorem{sub}[equation]{}

\newtheorem{prop-con}[equation]{Proposition--Construction}
\newtheorem{notation}[equation]{Notation}

\newtheorem{construction}[equation]{Construction}

\newtheorem{definition}[equation]{Definition}
\newtheorem{definition/lemma}[equation]{Definition--Lemma}

\theoremstyle{remark}
\newtheorem{remark}[equation]{Remark}
\title[Degenerating product of flag varieties]{Degenerating products of flag varieties and applications to the Breuil--M\'ezard conjecture}

\author{Robin Bartlett}
\date{\today}
\begin{document}
	
	\maketitle
	\begin{abstract}
		We consider closed subschemes in the affine grassmannian obtained by degenerating $e$-fold products of flag varieties, embedded via a tuple of dominant cocharacters. For $G= \operatorname{GL}_2$, and cocharacters small relative to the characteristic, we relate the cycles of these degenerations to the representation theory of $G$. We then show that these degenerations smoothly model the geometry of (the special fibre of) low weight crystalline subspaces inside the Emerton--Gee stack classifying $p$-adic representations of the Galois group of a finite extension of $\mathbb{Q}_p$. As an application we prove new cases of the Breuil--M\'ezard conjecture in dimension two.
	\end{abstract}
	\setcounter{tocdepth}{1}
	\tableofcontents
	
	\section{Introduction}
	
	\subsection*{Overview} Let $K$ be a finite extension of $\mathbb{Q}_p$ with residue field $k$ and let $\mathcal{X}_d$ denote the Emerton--Gee stack classifying $d$-dimensional $p$-adic representations of $G_K$. Inside $\mathcal{X}_d$ there are closed substacks $\mathcal{X}_d^{\mu,\tau}$ classifying potentially crystalline representations of type $(\mu,\tau)$, for $\mu$ and $\tau$ respectively Hodge and inertial types. When $\mu$ is regular (i.e. consists of distinct integers) these closed substacks have maximal dimension and the Breuil--M\'ezard conjecture \cite{BM02,EG14,EG19} predicts the existence of top dimensional cycles $\mathcal{C}_\lambda$ in the special fibre $\overline{\mathcal{X}}_d$ such that
	\begin{equation}\label{BMconj}
	[\overline{\mathcal{X}}_d^{\mu,\tau}] = \sum_\lambda m(\lambda,\mu,\tau) \mathcal{C}_\lambda	
\end{equation}

	where
	\begin{itemize}
		\item $\lambda$ runs over irreducible $\overline{\mathbb{F}}_p$-representations of $\operatorname{GL}_d(k)$.
		\item $m(\lambda,\mu,\tau)$ denotes the multiplicity with which $\lambda$ appears in an explicit $\mathbb{F}$-representation $V(\mu,\tau)$ of $\operatorname{GL}_d(k)$ attached to $\mu$ and $\tau$.
	\end{itemize}
	(there is also a version of the conjecture for substacks of potentially semistable representations; the conjecture has the same shape but with altered $V(\mu,\tau)$). These identities have been verified in only a small number of cases:
	\begin{enumerate}
		\item When $K= \mathbb{Q}_p$ and $d=2$, using the $p$-adic Langlands correspondence. See \cite{KisFM,Pas15,HT15,Sand,Tun21}.
		\item When $d=2$ and $\mu = (1,0)$, as consequence of certain modularity lifting theorems. See \cite{GK15}.
		\item When $K$ is unramified over $\mathbb{Q}_p$, $d$ is arbitrary, and both $p$ and $\tau$ are \emph{generic} relative to $\mu$. See \cite{LLBBM20}. Again modularity lifting technique play an important role.
	\end{enumerate}
In this paper we construct Breuil--M\'ezard identities in a fourth setting: we are interested in the two dimensional case where $\tau=1$ (i.e. we consider only crystalline rather than potentially crystalline representations) and $\mu$ is bounded so that the representation theory of $\operatorname{GL}_2(k)$ in the conjecture behaves as it does in characteristic zero. We do this by constructing analogous identities involving certain degenerations of products of flag varieties embedded in the affine grassmannian, and then relating the geometry of these degenerations to the geometry of the $\overline{\mathcal{X}}^{\lambda,1}_d$.

\subsection*{Main result} First we describe a bound on the Hodge types considered above, which is natural in the sense that the $\operatorname{GL}_d(k)$-representation theory appearing in the conjecture changes markedly once the bound is passed. Recall that a Hodge type $\mu$ consists of a $d$-tuple of integers 
$$
\mu_\kappa = (\mu_{\kappa,1} \geq \ldots \geq \mu_{\kappa,d})
$$
for each embedding $\kappa:K \rightarrow \overline{\mathbb{Q}}_p$. If one assumes that 
\begin{equation}\label{naturalbound}
	\sum_{\kappa|_{k} = \kappa_0} \mu_{\kappa,1}- \mu_{\kappa,d} \leq e+p-1
\end{equation}
for each embedding $\kappa_0:k \hookrightarrow \overline{\mathbb{F}}_p$ then:
\begin{itemize}
	\item $V(\mu,1)$ is a tensor product over the embeddings $\kappa$ of representations of highest weight $\mu_\kappa$ and the Jordan--Holder factors of this tensor product are computed in characteristic $p$ just as they are in characteristic zero, by Littlewood--Richardson coefficients.
	\item Each Jordan--Holder factor $\lambda$ of $V(\mu,1)$ can be written as $V(\widetilde{\lambda},1)$ for some Hodge type $\widetilde{\lambda}$ uniquely determined up to an ordering of the embeddings $\kappa$.
\end{itemize}
In particular, the cycles $\mathcal{C}_\lambda$ appearing in \eqref{BMconj} for these small $\mu$ are uniquely determined by the conjectured identity for $\mu = \widetilde{\lambda}$; one has $[\overline{\mathcal{X}}_d^{\widetilde{\lambda}}] = \mathcal{C}_{\lambda}$. Thus, the following theorem establishes new cases of the conjecture:

\begin{theorem}\label{A}
	Assume that $d=2$, $p>2$, $\mu$ is regular, and that
	$$
	\sum_{\kappa|_{k} = \kappa_0} \mu_{\kappa,1}- \mu_{\kappa,d} \leq p
	$$
	for each embedding $\kappa_0:k \rightarrow \overline{\mathbb{F}}_p$. Then
	\begin{equation}\label{BMconj2}
		[\overline{\mathcal{X}}_2^{\mu,1}] = \sum_\lambda m(\lambda,\mu,1) [\overline{\mathcal{X}}^{\widetilde{\lambda},1}_2]
	\end{equation}
\end{theorem}

There are some comments to make before we discuss what goes into the proof. Firstly, the theorem has two clear limitations: the assumption that $d=2$ and the fact that the bound on $\mu$ is stronger than that in \eqref{naturalbound} (we have no expectation whatever that the methods in this paper apply beyond \eqref{naturalbound}).

As we explain in more detail below, the proof of the theorem has two key inputs. The first involves relating the $\overline{\mathcal{X}}^{\mu,1}_d$ with certain local models we define inside the affine grassmannian. This can be done without any restriction on $d$ but our current argument requires the stronger bound on $\mu$. The second key input is a lower bound on the Breuil--M\'ezard identities which has been established when $d=2$ using global techniques from \cite{GK15} (this is also where the assumption $p>2$ appears).

Finally, taking $\widetilde{\lambda} = \mu$ shows that the cycle $[\overline{X}^{\widetilde{\lambda}}_2]$ is independent of the choice of ``lift'' $\widetilde{\lambda}$ of $\lambda$. We also show that each of these cycles consists of a single irreducible component occurring with multiplicity one.

\subsection*{Method}
The proof of the theorem divides into three parts:

\subsubsection*{Part 1: Local models in the affine grassmannian}

The starting point of the proof is the construction of certain projective schemes whose special fibres give upper bounds on the multiplicities appearing in Theorem~\ref{A}. To explain their construction we fix a sufficiently large extension $E$ of $\mathbb{Q}_p$, with ring of integers $\mathcal{O}$ and residue field $\mathbb{F}$, and consider a mixed characteristic version of the affine grassmannian $\operatorname{Gr}_{\mathcal{O}}$ over $\mathcal{O}$ whose special and generic fibres are given by
$$
\operatorname{Gr}_{\mathcal{O}} \otimes_{\mathcal{O}} \mathbb{F} \cong \prod_{\kappa_0:k\rightarrow \mathbb{F}} \operatorname{Gr} \otimes_{\mathcal{O}_K,\kappa_0} k, \qquad \operatorname{Gr}_{\mathcal{O}} \otimes_{\mathcal{O}} E \cong \prod_{\kappa:K\rightarrow E} \operatorname{Gr} \otimes_{\mathcal{O}_K,\kappa} E
$$
Here $\kappa_0$ and $\kappa$ are embeddings and $\operatorname{Gr}$ is the affine grassmannian over $\mathcal{O}_K$ whose $A$ points, for $A$ a $p$-adically complete $\mathcal{O}_K$-algebra, classify rank $d$-projective $A[[u]]$-modules satisfying
$$
(u-\pi)^aA[[u]]^d \subset \mathcal{E} \subset (u-\pi)^{-a}A[[u]]^d
$$
for some $a\in \mathbb{Z}_{\geq 0}$ and $\pi \in K$ a fixed choice of uniformiser. For each dominant cocharacter $\lambda$ of $G=\operatorname{GL}_d$ there is a closed immersion of the flag variety $G/P_\lambda \rightarrow \operatorname{Gr}$ ($P_\lambda \subset G$ being the parabolic corresponding to $\lambda$). This allows us to define, for any Hodge type $\mu = (\mu_\kappa)$, an $\mathcal{O}$-flat closed subscheme $M_\mu$ in $\operatorname{Gr}_{\mathcal{O}}$ by taking the closure in $\operatorname{Gr}_{\mathcal{O}}$ of
$$
\prod_\kappa (G/P_{\mu_\kappa} \otimes_{\mathcal{O}_K,\kappa} E) \hookrightarrow \prod_{\kappa} (\operatorname{Gr} \otimes_{\mathcal{O}_K,\kappa} E) = \operatorname{Gr}_{\mathcal{O}} \otimes_{\mathcal{O}} E
$$
The following summarises the key results we prove regarding these $M_\mu$'s

\begin{proposition}\label{prop1}
	Assume that $\mu$ is regular. 
	\begin{enumerate}
		\item If $\mu$ satisfies \eqref{naturalbound} then there exist $n(\lambda,\mu) \in \mathbb{Z}$ such that in the group of $\sum_\kappa \operatorname{dim}G/P_{\mu_\kappa}$-dimensional cycles
		$$
		[M_\mu \otimes_{\mathcal{O}} \mathbb{F}] = \sum_{\lambda} n(\lambda,\mu) [M_{\widetilde{\lambda}} \otimes_{\mathcal{O}} \mathbb{F}]
		$$
		with the sum running those irreducible $\operatorname{GL}_d(k)$-representations for which the Hodge type $\widetilde{\lambda}$ also satisfies \eqref{naturalbound}.
		\item If $d=2$ then the $M_{\widetilde{\lambda}} \otimes_{\mathcal{O}} \mathbb{F}$ appearing in (1) are irreducible, generically reduced, and produce pairwise distinct cycles. In particular, $n(\lambda,\mu) \geq 0$ in this case.
		\item If, for every $\lambda$, one has $n(\lambda,\mu) \geq m(\lambda,\mu,1)$ where $m(\lambda,\mu,1)$ denotes the multiplicity from the Breuil--M\'ezard conjecture, then $n(\lambda,\mu) = m(\lambda,\mu,1)$.
	\end{enumerate}
\end{proposition}

The first part is proved by constructing an explicit closed locus $\operatorname{Gr}^\nabla_{\mathcal{O}} \subset \operatorname{Gr}_{\mathcal{O}}$ defined in terms of a differential operator $\nabla$ (this is a variant of locus considered in \cite{LLBBM20}). A direct computation shows that if we bound the height according to \eqref{naturalbound} then the resulting closed subscheme of $\operatorname{Gr}^\nabla_{\mathcal{O}} \otimes_{\mathcal{O}} \mathbb{F}$ consists of irreducible components of dimension $\leq \operatorname{dim}M_\mu$. Furthermore, those components with maximal dimension are labelled by the $\lambda$'s appearing in (1). One can also show that $M_{\mu} \otimes_{\mathcal{O}} \mathbb{F}$ is contained in this closed subscheme. From these observations we are able to prove (1). 

\begin{remark}
	Unfortunately, this explicit moduli interpretation is only a good topological approximation of $M_\mu \otimes_{\mathcal{O}} \mathbb{F}$; typically the components appear with much too high multiplicity. 
\end{remark}

Part (2) is proved by constructing an explicit resolution of $X \rightarrow M_{\widetilde{\lambda}}$ with $X$ smooth and which is an isomorphism on the generic fibre. Unfortunately, we do not know how to construct such resolutions when $d>2$ (or whether they are likely to exist).

For part (3) we consider the restriction of the determinant line bundle on $\operatorname{Gr}_{\mathcal{O}}$ to $M_\mu$. Since the generic fibre of $M_\mu$ is a product of flag varieties it is easy to compute that for
\begin{equation}\label{globalssss}
H^0(M_\mu \otimes_{\mathcal{O}} E,\mathcal{L}_{\operatorname{det}}) = \bigotimes_\kappa H^0(\mu_\kappa) \otimes_{K,\kappa} E	
\end{equation}
where $H^0(\mu_\kappa)$ denotes the algebraic representation of $G$  over $K$ of highest weight $\mu_\kappa$. We point out that this tensor product differs from the $V(\mu,1)$ appearing in the Breuil--M\'ezard conjecture in that $V(\mu,1)$ is obtained as the reduction modulo $p$ of such a tensor product, but in which $\mu_\kappa$ is replaced by $\mu_\kappa - \rho$ for $\rho = (d-1,d-2,\ldots,1,0)$. Nevertheless, these multiplicities are approximately the same, in the sense that if, in the Grothendieck group of $E$-representations, one has
$$
[\bigotimes_{\kappa} H^0(\mu_\kappa- \rho) \otimes_{K,\kappa} E] = \sum_\lambda m(\lambda,\mu) [\bigotimes_\kappa H^0(\widetilde{\lambda}_\kappa- \rho) \otimes_{K,\kappa} E]
$$
then, for $n>0$,
$$
\operatorname{dim}\left( \bigotimes_{\kappa} H^0(n\mu_\kappa) \otimes_{K,\kappa} E \right) - \sum_\lambda m(\lambda,\mu) \operatorname{dim}\left(\bigotimes_\kappa H^0(n\widetilde{\lambda}_\kappa) \otimes_{K,\kappa} E \right)
$$
equals the value at $n$ of a polynomial of degree $< \operatorname{dim} M_\mu$. Since the representations $\bigotimes_{\kappa} H^0(n\mu_\kappa) \otimes_{\mathcal{O}_K,\kappa} E$ can be obtained by replacing $\mathcal{L}_{\operatorname{det}}$ with $\mathcal{L}_{\operatorname{det}}^{\otimes n}$ in \eqref{globalssss}, the identity of cycles in part (1) implies that
$$
\operatorname{dim}\left( \bigotimes_{\kappa} H^0(n\mu_\kappa) \otimes_{K,\kappa} E \right) - \sum_\lambda n(\lambda,\mu) \operatorname{dim}\left(\bigotimes_\kappa H^0(n\widetilde{\lambda}_\kappa) \otimes_{K,\kappa} E \right)
$$
is also the value of a polynomial in $n$ of degree $<\operatorname{dim} M_\mu$, at least for $n>>0$. Taking the difference shows that
\begin{equation}\label{dimdiffer}
\sum_\lambda (n(\lambda,\mu)-m(\lambda,\mu)) \operatorname{dim}\left(\bigotimes_\kappa H^0(n\widetilde{\lambda}_\kappa) \otimes_{K,\kappa} E \right)
\end{equation}
is polynomail in $n$ of degree $<\operatorname{dim}M_\mu$ for $n>>0$. For $\mu$ satisfying \eqref{naturalbound} the multiplicities $m(\lambda,\mu)$ computed in characteristic zero coincide with the $m(\lambda,\mu,1)$ computed in characteristic $p$. Thus, the assumption in (3) is that $n(\lambda,\mu)-m(\lambda,\mu) \geq 0$. Each term in \eqref{dimdiffer} is a polynomial in $n$ of degree $\operatorname{dim}M_\mu$ and positive leading term. Therefore we must have $n(\lambda,\mu)= m(\lambda,\mu)$.

\subsubsection*{Part 2: From local models to moduli of crystalline Galois representations} The second step is to relate the $M_\mu$'s to the geometry of $\mathcal{X}_d$. The basic strategy is to study the geometry of $\mathcal{X}_d$ via a resolution
$$
Y_d \rightarrow \mathcal{X}_d
$$
with $Y_d$ a stack whose $A$-points classify Breuil--Kisin modules with $A$-coefficients (i.e. projective $(W(k)\otimes_{\mathbb{Z}_p} A)[[u]]$-modules equipped with a semilinear endomorphism $\varphi$). A local version of this construction was first made in \cite{KisFF} (with $\mathcal{X}_d$ replaced by Spec of a deformation ring) and its globalisation to stacks first appeared in \cite{PR09}, before being built upon in \cite{EG19}.

In our case, we take $Y_d$ as the stack classifying pairs $(\mathfrak{M},\sigma)$ with $\mathfrak{M}$ a rank $d$ Breuil--Kisin module and $\sigma$ a $\varphi$-equivariant action of $G_K$ on $\mathfrak{M} \otimes_{W(k)[[u]]} A_{\operatorname{inf}}$ satisfying a ``crystalline'' condition (which means that $\sigma-1$ is sufficiently divisible). Inside $Y_d$ there are $\mathbb{Z}_p$-flat closed substacks $Y^{\mu}_d$ whose $\mathcal{O}$-valued points correspond to Breuil--Kisin modules associated to crystalline representations of Hodge type $\mu$ whenever $\mathcal{O}$ is the ring of integers in a finite extension of $\mathbb{Q}_p$. Then $\mathcal{X}_d^{\mu,1}$ is, by definition, the scheme theoretic image of the morphism $Y^\mu_d \rightarrow \mathcal{X}_d$.

To relate $Y_d$ to the affine grassmannian we use the following diagram:
\begin{equation}\label{locelmodel}
	Y_d \xleftarrow{\Gamma} \widetilde{Y}_d \xrightarrow{\Psi}  \operatorname{Gr}_{\mathcal{O}}
\end{equation}
Here $\widetilde{Y}_d$ classifies Breuil--Kisin modules in $Y_d$ together with a choice of basis (to stay in the world of finite type stacks this basis is taken  modulo $u^N$ for $N>>0$) and, using this choice of basis, the morphism $\Psi$ takes a Breuil--Kisin module $\mathfrak{M}$ to the relative position of $\mathfrak{M}$ and its image of Frobenius. The morphism $\Gamma$ forgets this choice of basis. The key result we prove is then

\begin{proposition}\label{prop2}
	\begin{enumerate} 
		\item	If $\mu$ satisfies the bound from Theorem~\ref{A} then the restriction of $\widetilde{Y}_d \rightarrow \operatorname{Gr}_{\mathcal{O}}$ to $\widetilde{Y}^\mu_d \otimes_{\mathcal{O}} \mathbb{F}$ (for $\widetilde{Y}^\mu_d$ the preimage of $Y^\mu_d$ in $\widetilde{Y}_d$) factors through $M_{-w_0\mu} \otimes_{\mathcal{O}} \mathbb{F}$ for $w_0 \in W$ the longest element.
		\item For such $\mu$, the morphism $\widetilde{Y}_d \rightarrow \operatorname{Gr}_{\mathcal{O}}$ is smooth over $M_\mu \otimes_{\mathcal{O}} \mathbb{F}$ with irreducible fibres of dimension equal the relative dimension of $\widetilde{Y}_d \rightarrow Y_d$.
	\end{enumerate}
\end{proposition}

To prove (1) it suffices to show this factorisation for $\overline{A}$-points for every finite $\mathbb{F}$-algebra $\overline{A}$. For simplicity, we sketch the argument only in the case where $\overline{A}= \mathbb{F}$. The general case requires only minor technical changes. We also assume $k = \mathbb{F}_p$ as this greatly simplifies the notation. If $e=1$ then $M_\mu=G/P_\mu$ is just a single flag variety and the claimed factorisation comes down to showing that for any $\overline{\mathfrak{M}} \in Y^{\mu,1}_d(\overline{A})$ and any basis $\overline{\beta}$ the module $\overline{\mathfrak{M}}$ is generated by
$$
\varphi(\overline{\beta})g\begin{pmatrix}
	u^{-\mu_1} & & \\ & \ddots & \\ & & u^{-\mu_d}
\end{pmatrix}
$$
for some $g \in \operatorname{GL}_d(\overline{A})$. This follows from results in \cite{GLS} where it is shown, for any lift of $\overline{\mathfrak{M}}$ to $\mathfrak{M} \in Y^\mu_d(A)$, with $A$ the ring of integers in a finite extension of $E$, and any basis $\beta$ that $(u-\pi)^p\mathfrak{M}$ is generated by
$$
\varphi(\beta)g\left[ \begin{pmatrix}
	(u-\pi)^{p-\mu_1} & & \\ & \ddots & \\ & & (u-\pi)^{p-\mu_d}
\end{pmatrix} + X_{\operatorname{err}}\right]
$$
for a matrix $X_{\operatorname{err}}$ divisible by a power of $\pi^{p -\mu_1+\mu_d+1}$ and $g \in \operatorname{GL}_d(A)$. Here $\pi \in K$ is a fixed uniformiser. This result does not directly extend to the case $e>1$. However, a variant of the method is able to show that, for each embedding $\kappa:K\rightarrow E$, the module $\mathfrak{M}^\varphi \cap (u-\kappa(\pi))^p \mathfrak{M}$ can be generated by
$$
\varphi(\beta)g_\kappa\left[ \begin{pmatrix}
	(u-\kappa(\pi))^{p-\mu_{\kappa,1}} & & \\ & \ddots & \\ & & (u-\kappa(\pi))^{p-\mu_{\kappa,d}}
\end{pmatrix} + X_{\operatorname{err},\kappa}\right]
$$
for some $g_\kappa \in \operatorname{GL}_d(A)$ and $X_{\operatorname{err},\kappa}$ a matrix divisible by $\pi^{p-\mu_{\kappa,1}+\mu_{\kappa,d}+1}$. This was done in \cite{GLS15} (actually they only consider the case $d=2$ but it is straightforward to extend their arguments to higher dimensions). If the $X_{\operatorname{err},\kappa}$'s are divisible by a high enough power of $\pi$ then it follows that
$$
\prod_\kappa (u-\kappa(\pi))^p \mathfrak{M} = \bigcap \left( \mathfrak{M}^\varphi \cap (u-\kappa(\pi))^p\mathfrak{M} \right) 
$$
is congruent modulo $\pi$ to the intersection of the submodules generated by 
$$
\varphi(\beta)g_\kappa \begin{pmatrix}
	(u-\kappa(\pi))^{p-\mu_{\kappa,1}} & & \\ & \ddots & \\ & & (u-\kappa(\pi))^{p-\mu_{\kappa,d}}
\end{pmatrix}
$$
This sufficient divisibility is ensured by the bound on $\mu$ from Theorem~\ref{A} and this congruence is precisely what it means of $\overline{\mathfrak{M}}$ to be mapped onto an element of $M_{-w_0\mu} \otimes_{\mathcal{O}} \mathbb{F}$ by $\Psi$. 

For the proof of (2) we factor the morphism $\Psi$ as $\widetilde{Y}_d \rightarrow \widetilde{Z}_d \rightarrow \operatorname{Gr}_{\mathcal{O}}$ where $\widetilde{Z}_d$ denotes the moduli stack of Breuil--Kisin modules (without a crystalline Galois action) and $\widetilde{Y}_d \rightarrow \widetilde{Z}_d$ forgets the Galois action. An easy calculation shows that over the special fibre $\widetilde{Z}_d \rightarrow \operatorname{Gr}_{\mathcal{O}}$ is smooth with irreducible fibres of dimension equal the relative dimension of $\widetilde{Y}_d \rightarrow Y_d$. Part (2) therefore reduces to understanding when $\widetilde{Y}_d \rightarrow \widetilde{Z}_d$ is an isomorphism. To address this we note that for any Breuil--Kisin module $\mathfrak{M}$ with basis $\beta$ we can define a naive Galois action $\sigma_{\operatorname{naive},\beta}$ on $\mathfrak{M}$ by semilinearly extending the trivial $G_K$-action on $\varphi(\beta)$. Usually $\sigma_{\operatorname{naive},\beta}$ will not be $\varphi$-equivariant or crystalline. However, we show that if $\sigma_{\operatorname{naive},\beta}-1$ is suitably divisible and if $\mathfrak{M}$ satisfies height conditions imposed by \eqref{naturalbound} (actually a very slight strengthening of this bound is required to avoid certain ``Steinberg'' situations) then 
$$
\operatorname{lim}_{n\rightarrow \infty} \varphi^n \circ \sigma_{\operatorname{naive},\beta} \circ \varphi^{-n}
$$
converges to a unique $\varphi$-equivariant  crystalline $G_K$-action. It turns out that the locus of $\widetilde{Z}_d$ on which $\sigma_{\operatorname{naive},\beta} -1$ is sufficiently divisible is closed, and obtained as the preimage of a closed subscheme in $\operatorname{Gr}_{\mathcal{O}}$. Part (2) is then proved by showing that $M_\mu \otimes_{\mathcal{O}} \mathbb{F}$ is contained in this closed subscheme.
\subsubsection*{Part 3: Upper and lower multiplicity bounds}

The final ingredient which goes into the proof of Theorem~\ref{A} is a lower bound on the multiplicities appearing in the Breuil--M\'ezard conjecture. The is the most critical place where we require $d=2$. It is also where we use that $p>2$. Under these assumptions it is shown in \cite[8.6]{EG19} (using global automorphy lifting techniques from \cite{GK15}) that one always has
$$
[\overline{\mathcal{X}}_2^{\mu,\tau}] \geq \sum_\lambda m(\lambda,\mu,\tau) \mathcal{C}_{\lambda}
$$
This holds without any assumption on $\mu$ or $\tau$. Combining \cite{GK15} with the potential diagaonalisability established in \cite{B19b} one also obtains that $\mathcal{C}_{\lambda} = [\overline{\mathcal{X}}^{\widetilde{\lambda},1}_2]$ so long as $\lambda$ is not Steinberg (for $d=2$ this means $\lambda$ is not a twist of $\bigotimes_{\kappa_0:k\rightarrow \mathbb{F}}\operatorname{Sym}^{p-1}\mathbb{F}^2$). The bounds on $\mu$ ensure Steinberg $\lambda$ do not appear in Theorem~\ref{A} (except if $K/\mathbb{Q}_p$ is unramified, but in this case the theorem is trivial). Therefore, for $\mu$ as in Theorem~\ref{A}, we have
$$
[\overline{\mathcal{X}}_2^{\mu,1}] \geq \sum_\lambda m(\lambda,\mu,1)[\overline{\mathcal{X}}_2^{\widetilde{\lambda},1}]
$$
To finish the proof we have to show that the results from parts 1. and 2. can be combined to give equality. 

First we consider the identity $[M_\mu \otimes_{\mathcal{O}} \mathbb{F}] = \sum_{\lambda} n(\lambda,\mu) [M_{\widetilde{\lambda}} \otimes_{\mathcal{O}} \mathbb{F}]$ from Proposition~\ref{prop1}. Applying an involution of $\operatorname{Gr}_{\mathcal{O}}$ which sends a lattice onto its dual allows us to replace $\mu$ and each $\widetilde{\lambda}$ in this identity with $-w_0\mu$ and $-w_0\widetilde{\lambda}$. Thus
$$
[M_{-w_0\mu} \otimes_{\mathcal{O}} \mathbb{F}] = \sum_{\lambda} n(\lambda,\mu) [M_{-w_0\widetilde{\lambda}} \otimes_{\mathcal{O}} \mathbb{F}]
$$
Part (2) of Proposition~\ref{prop2} ensures $\widetilde{Y}_2 \rightarrow \operatorname{Gr}_\mathcal{O}$ is smooth over the closed subschemes appearing in this identity of cycles. This allows us to pull the identity back to $\widetilde{Y}_d$ to obtain 
$$
[Y_2^{\mu,\operatorname{flag}}] = \sum_\lambda n(\lambda,\mu) [Y_2^{\widetilde{\lambda},\operatorname{flag}}]
$$
where $Y_2^{\mu,\operatorname{flag}}$ equals the preimage of $M_{-w_0\mu} \otimes_{\mathcal{O}} \mathbb{F}$ under this map. Part (1) of Proposition~\ref{prop2} (together with a dimension comparison) implies $[\widetilde{Y}_d^\mu \otimes_{\mathcal{O}} \mathbb{F}] \leq [Y_2^{\mu,\operatorname{flag}}]$. Using part (2) of Proposition~\ref{prop1} we are even able to deduce this is an equality when $\mu = \widetilde{\lambda}$. Therefore
$$
[\widetilde{Y}^\mu_2 \otimes_{\mathcal{O}} \mathbb{F}] \leq \sum_\lambda n(\lambda,\mu)[\widetilde{Y}^{\widetilde{\lambda}}_2 \otimes_{\mathcal{O}} \mathbb{F}]
$$
Since $\widetilde{Y}_d \rightarrow Y_d$ is smooth and surjective it follows that also
$$
[Y^\mu_2 \otimes_{\mathcal{O}} \mathbb{F}] \leq \sum_\lambda n(\lambda,\mu)[Y^{\widetilde{\lambda}}_2 \otimes_{\mathcal{O}} \mathbb{F}]
$$
Pushing this identity forward along the proper morphism $Y_2 \rightarrow \mathcal{X}_2$ gives an inequality $[\overline{\mathcal{X}}^\mu_2] \leq \sum_\lambda n(\lambda,\mu)[\overline{\mathcal{X}}_2^{\widetilde{\lambda}}]$. Combining this with the lower bound we obtain $n(\lambda,\mu) \geq m(\lambda,\mu,1)$. By part (3) of Proposition~\ref{prop1} this must be an equality, which proves Theorem~\ref{A}. Actually, in the paper we follow the same argument, but for the final step we prefer to work with deformation rings rather than $\overline{\mathcal{X}}_2$. This allows us to avoid dealing with stacks. As explained in \cite[8.3]{EG19}, Theorem~\ref{A} is implied by its analogue in the setting of deformation rings.

\subsection*{Acknowledgements}
I would like to thank Toby Gee, Eugen Hellmann, Brandon Levin and Timo Richarz for helpful conversations and correspondences. I would also like to thank the Max Plank Institute for Mathematics in Bonn, where parts of this work were done, for providing an excellent working environment. 
\section{Notation}

\begin{sub}\label{begin}
We fix a finite extension $K$ of $\mathbb{Q}_p$ with residue field $k$ of degree $f$ over $\mathbb{F}_p$ and ramification degree $e$. Let $C$ denote the completed algebraic closure of $K$, with ring of integers $\mathcal{O}_C$, and fix a compatible system $\pi^{1/p^\infty}$ of $p$-th power roots of a fixed choice of uniformiser $\pi \in K$ in $\mathcal{O}_C$. Set $K_\infty = K(\pi^{1/p^\infty})$. Write $E(u) \in W(k)[u]$ for the minimal polynomial of $\pi$. Thus $E(u)$ is Eisenstein of degree equal to the ramification degree $e$ of $K$ over $\mathbb{Q}_p$. 

We also fix another finite extension $E$ of $\mathbb{Q}_p$ with ring of integers $\mathcal{O}$ and residue field $\mathbb{F}$. We assume that $E$ contains a Galois closure of $K$. We typically use $\kappa$ and $\kappa_0$ respectively to denote embeddings $K\rightarrow E$ and $k\rightarrow \mathbb{F}$. For each $\kappa_0$ we fix an embedding $\widetilde{\kappa_0}:K\rightarrow E$ with $\widetilde{\kappa_0}|_k = \kappa_0$.
\end{sub}

\begin{sub}\label{mathfrak}
	For any $\mathbb{Z}_p$-algebra $A$ we write $\mathfrak{S}_A = (W(k)\otimes_{\mathbb{Z}_p} A)[[u]]$. This comes equipped with the $A$-linear endomorphism $\varphi$ which on $W(k)$ acts as the lift of the $p$-th power map on $k$ and sends $u \mapsto u^p$. We also consider
$$
A_{\operatorname{inf},A} = \varprojlim_a \varprojlim_i( W(\mathcal{O}_{C^\flat})/p^a \otimes_{\mathbb{Z}_p} A)/u^i
$$
where $\mathcal{O}_{C^\flat} = \varprojlim_{x\mapsto x^p} \mathcal{O}_C/p$ and $u = [(\pi,\pi^{1/p},\pi^{1/p^2},\ldots)] \in W(\mathcal{O}_{C^\flat})$. We view $A_{\operatorname{inf},A}$ as an $\mathfrak{S}_A$-algebra via $u$. Note that the lift of Frobenius on $W(\mathcal{O}_{C^\flat})$ induces a Frobenius $\varphi$ on $A_{\operatorname{inf},A}$ which is compatible with that on $\mathfrak{S}_A$. The natural $G_K$-action on $\mathcal{O}_C$ also induces a continuous (for the $(u,p)$-adic topology) $G_K$-action on $A_{\operatorname{inf},A}$ commuting with $\varphi$. Write 
$$
W(C^\flat)_A = \varprojlim_a A_{\operatorname{inf},A}[\tfrac{1}{u}]/p^a
$$
If $A$ is topologically of finite type (i.e. $A \otimes_{\mathbb{Z}_p} \mathbb{F}_p$ is of finite type) then $\mathfrak{S}_A \rightarrow A_{\operatorname{inf},A}$ is faithfully flat (in particular injective) \cite[2.2.13]{EG19}.

We also fix a compatible system $(1,\epsilon_1,\epsilon_2,\ldots)$ of $p$-th power roots of unity in $\mathcal{O}_C$ which we view as an element $\epsilon \in \mathcal{O}_{C^\flat}$. We write $\mu = [\epsilon] -1 \in A_{\operatorname{inf},A}$.
\end{sub}
\begin{lemma}\label{mu}
	Let $\Theta:A_{\operatorname{inf}} \rightarrow \mathcal{O}_C$ denote the surjection given by $\sum p^ix_i \mapsto \sum p^i x_i^{(0)} $ with $x_i = [(x_i^{(j)})_{j\geq 1}]$ and extend this to a surjection $A_{\operatorname{inf},\mathcal{O}} \rightarrow \mathcal{O} \otimes_{\mathbb{Z}_p} \mathcal{O}_C$. Then
	$$
	\mu A_{\operatorname{inf},\mathcal{O}} = \lbrace x \in A_{\operatorname{inf},\mathcal{O}} \mid \Theta(\varphi^n(x)) = 0 \text{ for all $n\geq 0$} \rbrace
	$$
\end{lemma}
\begin{proof}
	By choosing a $\mathbb{Z}_p$-basis of $\mathcal{O}$ this follows immediately from the assertion that $\mu A_{\operatorname{inf}} = \lbrace x \in A_{\operatorname{inf}} \mid \Theta(\varphi^n(x)) = 0 \text{ for all $n\geq 0$} \rbrace$ which is \cite{Fon94}[5.1.3].
\end{proof}
\begin{sub}\label{embeddings}
We frequently consider modules as in \ref{mathfrak} defined over $\mathcal{O} \otimes_{\mathbb{Z}_p} W(k)$ for an $\mathcal{O}$-algebra $A$. Using the isomorphism
$$
\mathcal{O} \otimes_{\mathbb{Z}_p} W(k) \xrightarrow{\sim} \prod_{\kappa_0:k\rightarrow \mathbb{F}} \mathcal{O} \cong \prod_{\kappa_0} \mathcal{O} \otimes_{W(k),\kappa_0} W(k)
$$
given by $a\otimes b \mapsto (a\kappa_0(b))_{\kappa_0}$ (here we write $\kappa_0$ to its extension to an embedding $W(k)\rightarrow \mathcal{O}$) we see that any such module $M$ can be expressed as a product
$$
M = \prod_{\kappa_0} M_{\kappa_0}
$$
where $M_{\kappa_0}$ can be identifies with the submodule of $M$ on which the two actions of $W(k)$ given by $ (1\otimes a)m$ and $a\mapsto (\kappa_0(a)\otimes 1)m$ coincide. Similarly, there is an isomorphism
\begin{equation}\label{prod1/p}
E \otimes_{\mathbb{Z}_p} \mathcal{O}_K \xrightarrow{\sim} \prod_{\kappa:K\rightarrow E} E \cong \prod_\kappa E \otimes_{\mathcal{O}_K,\kappa} \mathcal{O}_K	
\end{equation}
given by $a\otimes b \mapsto (\kappa(b)a)_\kappa$ which allows us to write an $E \otimes_{\mathbb{Z}_p} \mathcal{O}_K$-module $M$ as
$$
M = \prod_\kappa M_\kappa
$$
where again $M_\kappa$ can be identified with the submodule consisting of $m \in M$ with $(1\otimes a)m = (\kappa(a)\otimes 1)m$ for all $a \in \mathcal{O}_K$. We warn the reader that the idempotents in \eqref{prod1/p} will not be contained in $\mathcal{O} \otimes_{\mathbb{Z}_p} \mathcal{O}_K$ whenever $K/\mathbb{Q}_p$ is ramifies and so the product decomposition $M = \prod_\kappa M_\kappa$ is not valid integrally, i.e. when $M$ is an $\mathcal{O} \otimes_{\mathbb{Z}_p} \mathcal{O}_K$-module.
\end{sub}
\begin{sub}\label{Ekappa}
	Applying the previous discussion to $(A\otimes_{\mathbb{Z}_p} W(k))[u]$ allows us to write
	$$
	(A\otimes_{\mathbb{Z}_p} W(k))[u] = \prod_{\kappa_0} A[u]
	$$ 
	Using this identification we define $E_\kappa(u) \in (A\otimes_{\mathbb{Z}_p} W(k))[u]$ for every embedding $\kappa:K\rightarrow E$ as the element corresponding to
	$$
	(1,\ldots,1,u -\pi_\kappa,1,\ldots,1) \in \prod_{\kappa_0} A[u]
	$$
	where $\pi_\kappa:= \kappa(\pi)$ and the $u-\pi_\kappa$ appears in the $\kappa|_k$-th factor in the product. Notice that $E(u) = \prod_\kappa E_\kappa(u)$ inside $(A\otimes_{\mathbb{Z}_p} W(k))[u]$.
\end{sub}

\section{Cycles}

\begin{sub}
	For a Noetherian scheme $X$ let $Z_m(X)$ denote the free abelian group generated by integral closed subschemes $Z \subset X$ of dimension $m$. If $\mathcal{F}$ is a coherent sheaf on $X$ with support of dimension $\leq m$ then we define
$$
[\mathcal{F}] = \sum_Z \operatorname{length}_{\mathcal{O}_{X,\xi}} (\mathcal{F}_\xi) [Z] \in Z_m(X)
$$
for $\xi \in Z$ the generic point. If $i:Y\rightarrow X$ is a closed immersion write $[Y] = [i_*\mathcal{O}_Y]$. Any flat morphism $f:X \rightarrow Y$ of relative dimension $d$ produces a homomorphism $f^*: Z_m(Y) \rightarrow Z_{m+d}(X)$ with $f^*[\mathcal{F}] = [f^*\mathcal{F}]$. See \cite[02RE]{stacks-project}. If instead $f$ is proper then there is a pushforward homomorphism $f_*:Z_m(X)\rightarrow Z_m(Y)$ with $f_*[\mathcal{F}] = [f_*\mathcal{F}]$. See \cite[02R6]{stacks-project}.
\end{sub}
\begin{lemma}\label{hilbert}
Let $X$ be a projective scheme over $k$ equipped with an ample line bundle $\mathcal{L}$. Suppose that $Y,Y_1,\ldots,Y_s$ are $m$-dimensional closed subschemes in $X$ and that
$$
[Y] = \sum n_i [Y_i] \in Z_m(X)
$$
Then
$$
\operatorname{dim}H^0(Y,\mathcal{L}^{\otimes n}) - \sum n_i \operatorname{dim}H^0(Y_i,\mathcal{L}^{\otimes n})
$$
is, for large $n$, the value at $n$ of a polynomial of degree $<m$.
\end{lemma}
\begin{proof}
	This follows from \cite[0BEN]{stacks-project} and the fact that, since $\mathcal{L}$ is ample, the higher cohomologies of $\mathcal{L}^{\otimes n}$ vanish for $n>>0$ \cite[0B5U]{stacks-project}.
\end{proof}
\begin{lemma}\label{spec}
	Suppose that $f:X \rightarrow Y$ is a proper morphism between equidimensional flat $\mathcal{O}$-schemes which becomes an isomorphism after applying $\otimes_{\mathcal{O}} E$. Suppose $Z_X \subset X,Z_Y \subset Y$ are $\mathcal{O}$-flat top dimensional closed subschemes for which $f$ restricts to an isomorphism
	$$
	f: Z_X \otimes_{\mathcal{O}} E \xrightarrow{\sim} Z_Y \otimes_{\mathcal{O}} E
	$$ Then $f_*[Z_X \otimes_{\mathcal{O}} \mathbb{F}] = [Z_Y \otimes_{\mathcal{O}} \mathbb{F}]$.
\end{lemma}
\begin{proof}
	Let $A_m(X)$ denote the quotient of $Z_m(X)$ by rational equivalence. Since $X$ is $\mathcal{O}$-flat there is a specialisation homomorphism
	$$
	\sigma_X: A_m(X \otimes_{\mathcal{O}} E) \rightarrow A_m(X \otimes_{\mathcal{O}} \mathbb{F})
	$$
	with $\sigma_X([Z_X \otimes_{\mathcal{O}} E]) = [Z_X \otimes_{\mathcal{O}} \mathbb{F}]$ whenever $Z_X \subset X$ is a closed $\mathcal{O}$-flat subscheme of relative dimension $m$. Furthermore $\sigma_X$ commutes with proper pushforward. All this is explained in \cite[20.3]{Ful}.
	
	If $m = \operatorname{dim}X$ then $A_m(X) = Z_m(X)$ \cite[1.3.2]{Ful}. Therefore, the two stated properties of the specialisation map give
	$$
	f_*[Z_X \otimes_{\mathcal{O}} \mathbb{F}] = f_*\sigma_X([Z_X \otimes_{\mathcal{O}} E]) = \sigma_Y(f_*[Z_X \otimes_{\mathcal{O}} E]) = \sigma_Y([Z_Y \otimes_{\mathcal{O}} E]) = [Z_Y \otimes_{\mathcal{O}} \mathbb{F}]
	$$ 
	in $Z_{\operatorname{dim}Y \otimes_{\mathcal{O}} \mathbb{F}}(Y \otimes_{\mathcal{O}} \mathbb{F})$.
\end{proof}
\section{Local models}

\begin{sub}
	We begin by defining an ind-scheme $\operatorname{Gr}$ over $\mathcal{O}_K$ whose $A$-points classify rank $d$-projective $A[u]$-modules satisfying
\begin{equation}\label{Gr}
(u-\pi)^a A[u]^d \subset \mathcal{E} \subset 	\prod (u-\pi)^{-a} A[u]^d
\end{equation}
for some $a\geq 0$. For each $\kappa_0:k\rightarrow \mathbb{F}$, which we extend to an embedding $W(k) \rightarrow \mathcal{O}$, we also define $\operatorname{Gr}_{\kappa_0}$ as the ind-scheme over $\mathcal{O}$ whose $A$-points classify 
rank $d$-projective $A[u]$-modules satisfying
\begin{equation*}\label{Grkappa}
\kappa_0(E(u))^a A[u]^d \subset \mathcal{E} \subset 	\prod \kappa_0(E(u))^{-a} A[u]^d
\end{equation*}
for some $a\geq 0$. 

Note that, for each $\kappa:K\rightarrow E$, we can view an $A$-valued point of $\operatorname{Gr} \otimes_{\mathcal{O}_K,\kappa} \mathcal{O}$ as a rank $d$ projective $A[u]$-module satisfying $(u-\kappa(\pi))^a A[u]^d\subset \mathcal{E} \subset (u-\kappa(\pi))^{-a}A[u]^d$ for some $a\geq 0$. Therefore, if $\kappa|_k = \kappa_0$ then there is a natural closed embedding
$$
\operatorname{Gr} \otimes_{\mathcal{O}_K,\kappa} \mathcal{O} \rightarrow \operatorname{Gr}_{\kappa_0}
$$
\end{sub}
\begin{remark}
	Recall that an $A[u]$-submodule as in \eqref{Gr} is $A[u]$-projective of rank $d$ if and only if $(u-\pi)^{-a} A[u]^d/\mathcal{E}$ is $A$-projective. In particular, this illustrates the ind-representability of the functor; the locus of $\mathcal{E}$ as in \eqref{Gr} identifies with a closed subscheme of the usual grassmannian classifying submodules of $(u-\pi)^{-a}A[u]^d/(u-\pi)^aA[u]^d$.
\end{remark}
\begin{sub}\label{flags}
Write $X(T)$ for the group of characters of $\operatorname{GL}_d$ relative to $T$, the diagonal torus, and identify $X(T) = \mathbb{Z}^d$ as usual. We say an element $\mu = (\mu_1,\ldots,\mu_d) \in X(T)$ is dominant if $\mu_i \geq \mu_{i+1}$ and for any such dominant $\mu$ we write $\mathcal{E}_{\mu} \in \operatorname{Gr}_K$ for the $\mathcal{O}_K$-point generated by
$$
\begin{pmatrix}
	(u-\pi)^{\mu_1} & & \\ & \ddots & \\ & & (u-\pi)^{\mu_d}
\end{pmatrix}(e_1,\ldots,e_d)
$$
for $(e_1,\ldots,e_d)$ the standard basis in $\mathcal{O}_K[u]^d$. There is an obvious action of $G = \operatorname{GL}_d$ on $\operatorname{Gr}$ and, since the stabiliser of $\mathcal{E}_{\mu}$ under this action is a parabolic subgroup $P_{\mu} \subset G$, the orbit map induces a proper monomorphism
$$
G/P_{\mu} \rightarrow \operatorname{Gr}
$$
i.e. a closed immersion. If we interpret the $A$-points of $G/P_{\mu}$ as filtrations 
$$
 \ldots \subset \operatorname{Fil}^{n+1}\subset \operatorname{Fil}^n \subset \operatorname{Fil}^{n-1}\subset \ldots
$$ 
of type $\mu$ on $A^d$ (which means the $n$-th graded piece is $A$-projective of rank equal to the multiplicity of $-n$ in $\mu$) then on $A$-points this closed immersion is given by
$$
\operatorname{Fil}^\bullet \mapsto \sum_{i \geq\lambda_d} (u-\pi)^{i} A[u]\operatorname{Fil}^{-i}
$$ 
where we view $A^d$ as a submodule of $A[u]^d$ in the obvious way.
\end{sub}

\begin{lemma}\label{beueville}
	If $A$ is a $p$-adically complete $\mathcal{O}$-algebra then $\operatorname{Gr}$-identifies with the set of rank $d$ projective $A[[u]]$-modules satisfying
	$$
	(u-\pi)^aA[[u]] \subset \mathcal{E} \subset(u-\pi)^{-a}A[[u]]
	$$
	for some $a\geq 0$. Similarly, for each $\operatorname{Gr}_{\kappa_0}$.
\end{lemma}
\begin{proof}
	This follows from the Beauville--Laszlo gluing lemma \cite{BL}.
\end{proof}

\begin{lemma}\label{generic} For each $\kappa_0:k\rightarrow \mathbb{F}$ there is an isomorphism 
	$$
	\operatorname{Gr}_{\kappa_0} \otimes_{\mathcal{O}} E \rightarrow \prod_{\kappa|_k = \kappa_0} \left( \operatorname{Gr} \otimes_{\mathcal{O}_K,\kappa} E \right)
	$$
	with inverse given by $(\mathcal{E}_{\kappa}) \mapsto \bigcap_{\kappa|_k = \kappa_0} \mathcal{E}_{\kappa}$.
\end{lemma}
\begin{proof}
Let $U \subset \mathbb{A}^1_A$ denote the open obtained by inverting $\kappa_0(E(u))$ and write $U_\kappa \subset \mathbb{A}^1_A$ for the open obtained by inverting $(u-\kappa'(\pi))$ for each $\kappa' \neq \kappa$ with $\kappa'|_k = \kappa_0$. Then $U = \bigcap U_\kappa$ and if $A$ is an $E$-algebra then the $U_\kappa$ form an open cover of $\mathbb{A}^1_A$.

Note that an $A$-valued point of $\operatorname{Gr}_{\kappa_0}$ is the same thing as a rank $d$ vector bundle on $\mathbb{A}^1_A$ which is trivial over $U$ while an $A$-valued point of $\operatorname{Gr} \otimes_{\mathcal{O}_K,\kappa} \mathcal{O}$ is likewise vector bundle trivial over $\bigcup_{\kappa' \neq \kappa} U_\kappa$. 

The map in the lemma can therefore be expressed as $\mathcal{E} \mapsto (\mathcal{E}_\kappa)$ where $\mathcal{E}_{\kappa}$ is the vector bundle obtained by glueing $\mathcal{E}|_{U_\kappa}$ with the trivial bundle on $\bigcup_{\kappa'\neq \kappa} U_{\kappa'}$. The inverse of this map sends $(\mathcal{E}_{\kappa})$ onto the vector bundle obtained by glueing the $\mathcal{E}_{\kappa}|_{U_\kappa}$. Concretely, this glueing corresponds to taking the intersection of each of the $\mathcal{E}_{\kappa}$'s which gives the lemma.
\end{proof}

\begin{sub}
	We define one last ind-scheme $\operatorname{Gr}_{\mathcal{O}}$ whose $A$-points now classify rank $d$ projective $(A\otimes_{\mathbb{Z}_p} W(k))[u]$-modules satisfying
	$$
	E(u)^a(A\otimes_{\mathbb{Z}_p} W(k))[u]^d \subset \mathcal{E} \subset E(u)^{-a}(A\otimes_{\mathbb{Z}_p} W(k))[u]^d
	$$
	for some $a\geq 0$. From \ref{embeddings} we see that
	$$
	\operatorname{Gr}_{\mathcal{O}} \cong \prod_{\kappa_0} \operatorname{Gr}_{\kappa_0}
	$$
	Lemma~\ref{generic} implies that the generic fibre of $\operatorname{Gr}_{\mathcal{O}}$ identifies with $\prod_{\kappa} \left( \operatorname{Gr}\otimes_{\mathcal{O}_K,\kappa} E \right)$ with the product running over all embeddings $\kappa:K \rightarrow E$. Note also that the analogue of Lemma~\ref{beueville} applies to $\operatorname{Gr}_{\mathcal{O}}$ and identifies its points valued in $p$-adically complete $\mathcal{O}$-algebras $A$ with rank $d$ projective $\mathfrak{S}_A$-modules satisfying 
	$$
	E(u)^{a}\mathfrak{S}_A^d \subset \mathcal{E} \subset E(u)^{-a}\mathfrak{S}_A^d
	$$
\end{sub}
\begin{definition}
	Let $\mu = (\mu_\kappa)$ be a Hodge type, i.e. a collection of dominant $\mu_\kappa \in X(T)$ indexed by embeddings $\kappa:K \rightarrow E$. Then we define $M_\mu$ as the closure in $\operatorname{Gr}_{\mathcal{O}}$ of
	$$
	\prod_{\kappa} (G/P_{\mu_\kappa} \otimes_{\mathcal{O}_K,\kappa} E) \hookrightarrow \prod_{\kappa} \left( \operatorname{Gr}\otimes_{\mathcal{O}_K,\kappa} E \right) \cong \operatorname{Gr}_{\mathcal{O}} \otimes_{\mathcal{O}} E
	$$
\end{definition}

\begin{lemma}\label{points}
Let $\mu$ be a Hodge type and suppose $n_\kappa \geq 0$ so that $\mu_{\kappa,d} \geq -n_\kappa$ for every $\kappa$. 
\begin{enumerate}
	\item Let $A$ be an $E$-algebra. Then $\mathcal{E} \in \operatorname{Gr}_{\mathcal{O}}(A)$ is contained in $M_\mu$ if and only if there are filtrations $\operatorname{Fil}_{\kappa}^\bullet$ on $A^d$ of type $\kappa$ so that
	$$
	\left( \prod_\kappa E_\kappa(u)^{n_\kappa} \right) \mathcal{E} = \bigcap_\kappa \left( \sum_{i\geq \mu_{\kappa,d}+n_{\kappa}} E_{\kappa}(u)^{i+n_\kappa} (A\otimes_{\mathbb{Z}_p} W(k))[u] \operatorname{Fil}^{-i}_\kappa \right)
	$$
	(recall the elements $E_\kappa(u)$ from \ref{embeddings}).
	\item Let $A$ be a $p$-adically complete Noetherian flat $\mathcal{O}$-algebra and suppose there are $A$-submodules 
	$$
	\ldots \subset \operatorname{Fil}^{i+1}_{\kappa }\subset \operatorname{Fil}^i_{\kappa} \subset \operatorname{Fil}^{i-1}_{\kappa} \subset \ldots \subset A^d
	$$
	for each $\kappa$ such that $\operatorname{Fil}^i_\kappa/\operatorname{Fil}^{i+1}_\kappa$ is $p$-torsionfree and becomes $A[\frac{1}{p}]$-projective of constant rank after inverting $p$. If $\mathcal{E} \in \operatorname{Gr}(A)$ can be expressed as 
	$$
	\left( \prod_\kappa E_\kappa(u)^{n_\kappa} \right)\mathcal{E} = \bigcap_\kappa \left( \sum_{i\geq \mu_{\kappa,d}+n_\kappa} E_{\kappa}(u)^{i+n_\kappa} \mathfrak{S}_A \operatorname{Fil}^{-i}_\kappa \right)
	$$
	 then $\mathcal{E} \in M_{\mu}(A)$ for $\mu_\kappa$ the type of $\operatorname{Fil}_\kappa[\frac{1}{p}]^\bullet$.
	\item If $A$ is the ring of integers in a finite extension of $E$ then every $A$-valued point of $M_\mu$ is as in (2).
	\end{enumerate}	
\end{lemma}
\begin{proof}
	Note that multiplication by $\left( \prod_\kappa E_\kappa(u)^{n_\kappa} \right)$ identifies $M_{\mu}$ with $M_{\mu'}$ for $\mu'_\kappa = \mu_\kappa + (n_\kappa)$. Thus we can assume $n_\kappa=0$ throughout.
	
	For (1) we first decompose $\mathcal{E} = \prod_{\kappa_0} \mathcal{E}_{\kappa_0} \in \prod_{\kappa_0} \operatorname{Gr}_{\kappa_0}$ according to the action of $W(k)$. Then Lemma~\ref{generic} and the description of $G/P_{\mu_\kappa} \hookrightarrow \operatorname{Gr}$ from \ref{flags} implies $\mathcal{E} \in M_\mu$ if and only if, for each $\kappa_0$,
	$$
	\mathcal{E}_{\kappa_0} = \bigcap_{\kappa_{\kappa|_k =\kappa_0}} \left( \sum_{i \geq \lambda_d} (u-\kappa(\pi))^iA[u] \operatorname{Fil}^{-i}_\kappa \right)
	$$
	for filtrations $\operatorname{Fil}^\bullet_\kappa$ on $A^d$ of type $\mu_\kappa$. Since $\lambda_d \geq 0$ we have $\mathcal{E}_{\kappa_0} \subset A[u]^d$ for each $\kappa_0$ and so 
	$$
	\mathcal{E} = \bigcap_{\kappa_0} \left( A[u]^d \times \ldots \times A[u]^d \times \underbrace{\mathcal{E}_{\kappa_0}}_{\kappa_0\text{-th position}} \times A[u]^d \times \ldots \times A[u]^d \right) 
	$$
 	Thus, to prove (1) we just need to identify the $\kappa_0$-th term inside this intersection with $ \bigcap_{\kappa_{\kappa|_k =\kappa_0}} \left( \sum_{i\geq \mu_{\kappa,d}} E_{\kappa}(u)^{i} (A \otimes_{\mathbb{Z}_p} W(k))[u] \operatorname{Fil}^{-i}_\kappa\right)$. This is clear since $E_\kappa(u)$ corresponds to $(1,\ldots,1,(u-\kappa(\pi)),1,\ldots,1)$ under the identification $(\mathcal{O}\otimes_{\mathbb{Z}_p} W(k))[u] \cong \prod_{\kappa_0} A[u]$.

For part (2) we use that $A$ is Noetherian to ensure $(A\otimes_{\mathbb{Z}_p} W(k))[u] \rightarrow \mathfrak{S}_A$ is flat. Thus $\otimes_{(A\otimes_{\mathbb{Z}_p} W(k))[u]} \mathfrak{S}_A$ commutes with finite intersections and so 
$$
\mathcal{E} = \left( \bigcap_\kappa \left( \sum_{i\geq \mu_{\kappa,d}} E_{\kappa}(u)^{i+n} (A\otimes_{\mathbb{Z}_p} W(k))[u]\operatorname{Fil}^{-i}_\kappa \right) \right)\otimes_{(A\otimes_{\mathbb{Z}_p} W(k))[u]} \mathfrak{S}_A
$$
As a consequence of (1) it follows that $\mathcal{E}[\frac{1}{p}] \in M_{\mu}$ and so $\mathcal{E} \in M_\mu$ also.

Part (3) relies on the fact that, for $A$ as in the proposition, being $A$-projective is equivalent to being $p$-torsion free and finitely generated. Applying (1) to $\mathcal{E}[\frac{1}{p}]$ produces filtrations on $A[\frac{1}{p}]^d$ for each $\kappa$. If $\operatorname{Fil}^i_\kappa$ denotes the intersection of this filtration with $A^d$ then the graded pieces are $p$-torsionfree. This is equivalent to asking that each $\mathfrak{S}_A^d / \left( \sum_{i\geq \mu_{\kappa,d}} E_{ij}(u)^{i+n} \mathfrak{S}_A \operatorname{Fil}^{-i}_\kappa \right)$
is $p$-torsionfree. Therefore $\mathfrak{S}_A^d / \bigcap_\kappa \left( \sum_{i\geq \mu_{\kappa,d}} E_{ij}(u)^{i+n} \mathfrak{S}_A \operatorname{Fil}^{-i}_\kappa \right)$ is also $p$-torstionfree, and so
$$
\mathcal{E}' = \bigcap_\kappa \left( \sum_{i\geq \mu_{\kappa,d}} E_{ij}(u)^{i+n} \mathfrak{S}_A \operatorname{Fil}^{-i}_\kappa \right)
$$
is $\mathfrak{S}_A$-projective and $\mathcal{E}' \in \operatorname{Gr}_{\mathcal{O}}$. Since $\mathcal{E}'[\frac{1}{p}] = \mathcal{E}[\frac{1}{p}]$ the valuative criterion for properness implies $\mathcal{E}' = \mathcal{E}$.
\end{proof}
\section{Cohomology}

\begin{sub}\label{linebundle}
	Recall $G = \operatorname{GL}_d$ viewed as an algebraic group over $\mathcal{O}_K$. Let $\lambda \in X(T)$ be dominant and set 
	$$
	\ldots \subset \operatorname{Fil}^{n+1} \subset \operatorname{Fil}^n \subset \operatorname{Fil}^{n-1} \subset \ldots \subset \mathcal{O}_{G/P_\lambda}^d
	$$
	equal to the universal filtration on $G/P_{\lambda}$ of type $\lambda$. Then 
	$$
	\mathcal{L}(\lambda) := \bigotimes_n \operatorname{det}(\operatorname{Fil}^{-n}/\operatorname{Fil}^{-n+1})^{ \otimes n}
	$$
	is a $G$-equivariant line bundle on $G/P_\lambda$ and $H^0(G/P_{\lambda},\mathcal{L}(\lambda)^{\otimes n})$ can be viewed as an algebraic representation of $G$ on a flat $\mathcal{O}_K$-module whose generic fibre identifies with $H^0(n\lambda)$, the algebraic representation over $K$ of highest weight $n\lambda$. See for example \cite[p.143-144]{FulY}.
\end{sub}
\begin{sub}
	On $\operatorname{Gr}$ there is an ample $G$-equivariant line bundle $\mathcal{L}_{\operatorname{det}}$ whose fibre over any $A$-valued point $\mathcal{E} \in \operatorname{Gr}(A)$ with $\mathcal{E} \subset (u-\pi)^{-a}\mathfrak{S}_A^d$ for $a \geq 0$ is given by
	$$
	\operatorname{det}(\mathcal{E}_0/\mathcal{E}) \otimes \operatorname{det}(\mathcal{E}_0/A[u]^d)^{-1}
	$$
	for any $\mathcal{E}_0 \in \operatorname{Gr}$ with $\mathcal{E},A[u]^d \subset \mathcal{E}_0$. Note this is $G$-equivariantly independent of $\mathcal{E}_0$. We also write $\mathcal{L}_{\operatorname{det}}$ for the $G$-equivariant ample line bundle constructed analogously on $\operatorname{Gr}_{\mathcal{O}}$.
\end{sub}
\begin{lemma}\label{equal}
	The restriction of $\mathcal{L}_{\operatorname{det}}$ to $G/P_{\lambda}$ inside $\operatorname{Gr}$ identifies $G$-equivariantly with $\mathcal{L}(\lambda)$.
\end{lemma}
\begin{proof}
	Suppose $\mathcal{E} \in \operatorname{Gr}_K$ corresponds to an $A$-valued point in the image of $G/P_{\lambda} \rightarrow \operatorname{Gr}_K$. Then $\mathcal{E} = \sum_{i \geq \lambda_d} (u-\pi)^{i}A[u]\operatorname{Fil}^{-i}$ for a filtration $\operatorname{Fil}^\bullet$ of type $\lambda$ on $A^d$ and so, as $A$-modules, $(u-\pi)^{\lambda_d}A[u]/\mathcal{E} = \bigoplus_{i\geq \lambda_d} A^d/\operatorname{Fil}^{-i}$.
	Thus,
	$$
	\operatorname{det}\left((u-\pi)^{\lambda_d}A[u]/\mathcal{E}\right) = \bigotimes_{i \geq \lambda_d} \operatorname{det}(\operatorname{Fil}^{-i}/\operatorname{Fil}^{-i+1})^{\otimes (i-\lambda_d)}
	$$
	and so the fibre of $\mathcal{L}_{\operatorname{det}}$ over $\mathcal{E}$ equals
	$$
	\begin{cases}
		\bigotimes_{i \geq \lambda_d} \operatorname{det}(\operatorname{Fil}^{-i}/\operatorname{Fil}^{-i+1})^{\otimes (i-\lambda_d)} \otimes \operatorname{det}((A[u]^d/(u-\pi)^{\lambda_d}A[u]^d)) & \text{if $\lambda_d \geq 0$}\\
		
		\bigotimes_{i \geq \lambda_d} \operatorname{det}(\operatorname{Fil}^{-i}/\operatorname{Fil}^{-i+1})^{\otimes (i-\lambda_d)} \otimes \operatorname{det}(((u-\pi)^{\lambda_d}A[u]^d/A[u]^d))^{-1} & \text{if $\lambda_d \geq 0$}
	\end{cases} 
	$$
	In either case, the second factor in these tensor products identifies with $\bigotimes_{i\geq \lambda_d} \operatorname{det}(\operatorname{Fil}^{-i}/\operatorname{Fil}^{-i+1})^{\lambda_d}$ which finishes the proof.
\end{proof}

\begin{corollary}\label{globalsec}
	For any $n >0$, there is an identification
	$$
	H^0(M_{\mu} \otimes_{\mathcal{O}} E,\mathcal{L}_{\operatorname{det}}^{\otimes n}) = \bigotimes_{\kappa} (H^0(n \mu_\kappa) \otimes_{K,\kappa} E)
	$$
	of $G$-representations.
\end{corollary}
\begin{proof}
	Let $p_\kappa: M_{\mu} \otimes_{\mathcal{O}} E \rightarrow G/P_{\mu_\kappa} \otimes_{\mathcal{O}_K,\kappa} E$ be the $\kappa$-th projection. Then the restriction of $\mathcal{L}_{\operatorname{det}}$ on $\operatorname{Gr}_{\mathcal{O}}$ to $M_\mu \otimes_{\mathcal{O}} E$ coincides with $\bigotimes_{\kappa} p_\kappa^* (\mathcal{L}_{\operatorname{det}} \otimes_{\mathcal{O}_K,\kappa} E)$ where $\mathcal{L}_{\operatorname{det}}$ here denotes the restriction to $G/P_{\mu_\kappa}$ of the determinant line bundle on $\operatorname{Gr}$. The Kunneth formula \cite[0BED]{stacks-project} gives
	$$
	H^0(M_{\mu} \otimes_{\mathcal{O}} E,\mathcal{L}_{\operatorname{det}}^{\otimes n}) = \bigotimes_{\kappa} H^0(G/P_{\mu_\kappa},\mathcal{L}_{\operatorname{det}}^{\otimes n}) \otimes_{\mathcal{O}_K,\kappa} E
	$$
	as $G$-representations. Therefore, we just have to show $H^0(G/P_{\mu_\kappa},\mathcal{L}_{\operatorname{det}}^{\otimes n}) \otimes_{\mathcal{O}_K} K = H^0(n \mu_\kappa)$ as $G$-representations, and this follows from Lemma~\ref{equal}.
\end{proof}
\section{Multiplicity bounds}
\begin{sub}
	The formal character of an algebraic representation $V$ of $G$ on a finite dimensional vector space is defined as
	$$
	\operatorname{ch}(V) = \sum_{\lambda} V_\lambda e(\lambda) \in \mathbb{Z}[X(T)]
	$$
	where $V_{\lambda}$ is the $\lambda$-weight space of $V$ and $e(\lambda)$ denotes $\lambda$ viewed as an element of the group ring $\mathbb{Z}[X(T)]$. This induces an isomorphism between the Grothendieck group of such representations and $\mathbb{Z}[X(T)]^W$ where $W$ denotes the Weyl group of $G$ \cite[II.5.7]{Janbook}.
\end{sub}
\begin{sub}\label{Weyl}
	For $\lambda \in X(T)$ dominant recall the $G$-representation $H^0(\lambda)$ over $K$ from \ref{linebundle}. Weyl's character formula gives
	$$
	\operatorname{ch}(H^0(\lambda))  = \frac{A(\lambda+\rho)}{A(\rho)},\qquad \rho := (d-1,d-2,\ldots,1,0) \in X(T)
	$$
	where $A(\lambda) := \sum_{w\in W} \operatorname{det}(w) e(w\lambda)$ \cite[II.5.10]{Janbook}. If we write $\operatorname{dim}:\mathbb{Z}[X(T)] \rightarrow \mathbb{Z}$ for the map $\sum_{\lambda} a_\lambda e(\lambda) \mapsto \sum a_{\lambda}$ then one also has
	$$
	\operatorname{dim}\operatorname{ch}(H^0(\lambda)) = \operatorname{dim} H^0(\lambda) = \prod_{i>j} \frac{\lambda_j-\lambda_i +i-j}{i-j}
	$$
	Though here $H^0(\lambda)$ is defined over a field of characteristic zero, all the above goes through with $H^0(\lambda)$ replaced by the representation over a field of characteristic $p$ of highest weight $\lambda$. What differs in characteristic $p$ is that this highest weight representation may not be irreducible.
\end{sub}
\begin{lemma}\label{polynomialsss}
	Let $\mu_1,\ldots,\mu_e \in X(T)$ with $\mu_i -\rho$ dominant for each $i$ and suppose that
	$$
	\operatorname{ch}\left( \bigotimes_{i=1}^e H^0(\mu_i-\rho) \right) = \sum_{\lambda \in X(T)} m(\lambda,\mu) \operatorname{ch}H^0(\lambda)
	$$
	for $m(\lambda,\mu) \in \mathbb{Z}$. Then
	$$
	\operatorname{dim}\left( \bigotimes_{i=1}^e H^0(n\mu_i) \right) - \sum_{\lambda \in X(T)} m(\lambda,\mu) \operatorname{dim}H^0(n(\lambda+\rho)) \operatorname{ch}\left( H^0(n\rho)^{\otimes (e-1)} \right)
	$$
	is a polynomial in $n$ of degree $<\sum_{i} \operatorname{dim} G/P_{\mu_i}$.
\end{lemma}
\begin{proof}
	Using Weyls character formula from  \ref{Weyl} and multiplying by $A(\rho)^e$ gives $\prod_i A(\mu_i) = \sum_\lambda m(\lambda,\mu) A(\lambda+\rho) A(\rho)^{e-1}$. Taking the image of this identity under the endomorphism of $\mathbb{Z}[X(T)]$ induced by multiplication by $n$ on $X(T)$ gives
	$$
	\prod_i A(n\mu_i) = \sum_{\lambda} m(\lambda,\mu) A(n(\lambda+\rho))A(n\rho)^{e-1}
	$$
	(because the formation of $A$ commutes with this endomorphism). Dividing by $A(\rho)^e$ then gives 
	$$
	\prod_i \operatorname{ch}(H^0(n\mu_i-\rho) )= \sum_{\lambda} m(\lambda,\mu) \operatorname{ch}(H^0(n(\lambda+\rho)-\rho)) \left( \operatorname{ch}(H^0(n\rho-\rho))\right) ^{e-1}
	$$
	The lemma therefore follows by taking the dimension and observing that $\operatorname{dim}H^0(n\lambda) - \operatorname{dim}H^0(n\lambda -\rho)$ is a polynomial in $n$ of degree $<\operatorname{dim}G/P_\lambda$ for any dominant $\lambda \in X(T)$ (use the last equation from \ref{Weyl}).
\end{proof}
\begin{remark}
	Since $K$ has characteristic zero each $H^0(\lambda)$ is irreducible. Moreover, every irreducible $G$-representation is isomorphic to one such $H^0(\lambda)$. The observation from \ref{Weyl} that $\operatorname{ch}$ induces an identification between $\mathbb{Z}[X(T)]$ and the Grothendieck group of $G$-representations shows that the integers $m(\lambda,\mu)$ in the previous lemma are $\geq 0$ and are uniquely determined.
\end{remark}
\begin{notation}\label{tilde}
	Recall the fixed elements $\widetilde{\kappa}_0:K \rightarrow E$ lifting the $\kappa_0$ from \ref{begin}. Then, for any tuple $\lambda = (\lambda_{\kappa_0})_{\kappa_0:k\rightarrow \mathbb{F}}$ of dominant $\lambda_{\kappa_0} \in X(T)$ write $\widetilde{\lambda} = (\widetilde{\lambda}_\kappa)$ for the Hodge type defined by 
	$$
	\widetilde{\lambda}_{\kappa} = \begin{cases}
		\lambda_{\kappa_0} + \rho & \text{if $\kappa = \widetilde{\kappa}_0$} \\ 
		\rho & \text{otherwise}
	\end{cases}
	$$
\end{notation}
\begin{proposition}\label{upperbound}
	Let $\mu$ be a Hodge type with $\mu_{\kappa}-\rho$ dominant for every $\kappa$ and suppose that
	$$
	[M_\mu \otimes_{\mathcal{O}} \mathbb{F}] = \sum n(\lambda,\mu) [M_{\widetilde{\lambda}} \otimes_{\mathcal{O}} \mathbb{F}]
	$$
	for integers 
	$$
	n(\lambda,\mu) \geq \prod_{\kappa_0} m(\lambda_{\kappa_0},(\mu_\kappa)_{\kappa|_k =\kappa_0})
	$$
	Then this inequality is an equality for each $\lambda$.
\end{proposition}
\begin{proof}
We apply Lemma~\ref{hilbert} to the line bundle $\mathcal{L}_{\operatorname{det}}$. This gives that
$$
\operatorname{dim}H^0(\overline{M}_\mu,\mathcal{L}_{\operatorname{det}}^{\otimes n}) - \sum_\lambda n(\lambda,\mu) \operatorname{dim}H^0(\overline{M}_{\widetilde{\lambda}},\mathcal{L}_{\operatorname{det}}^{\otimes n})
$$
is, for large $n$, equal the value at $n$ of a polynomial of degree $< \sum_{\kappa} \operatorname{dim} G/P_{\mu_\kappa}$. Applying Corollary~\ref{globalsec} implies the same is true for
$$
\prod_{\kappa} \operatorname{dim}( H^0(n\mu_\kappa) ) -\sum_{\lambda = (\lambda_{\kappa_0})} n(\lambda,\mu)\prod_{\kappa_0}\left(  \operatorname{dim}( H^0(n(\lambda_{\kappa_0}+\rho)))\operatorname{dim}( H^0(n\rho))^{e-1} \right)
$$
Set $m_\lambda = \prod_{\kappa_0} m(\lambda_{\kappa_0},(\mu_\kappa)_{\kappa|_k =\kappa_0})$. Lemma~\ref{polynomialsss} gives that
$$
\prod_{\kappa} \operatorname{dim}( H^0(n\mu_\kappa) ) -\sum_{\lambda = (\lambda_{\kappa_0})} m_\lambda\prod_{\kappa_0}\left(  \operatorname{dim}( H^0(n(\lambda_{\kappa_0}+\rho)))\operatorname{dim}( H^0(n\rho))^{e-1} \right)
$$
is a polynomial of degree $< \sum_{\kappa} \operatorname{dim} G/P_{\mu_\kappa}$ in $n$. We conclude that the dimension of
$$
\sum_{\lambda = (\lambda_{\kappa_0})} \left( n(\lambda,\mu)-m(\lambda,\mu)\right)\prod_{\kappa_0}\left(  \operatorname{dim}( H^0(n(\lambda_{\kappa_0}+\rho)))\operatorname{dim}( H^0(n\rho))^{e-1} \right)
$$
is also polynomial of degree $< \sum_{\kappa} \operatorname{dim} G/P_{\mu_\kappa}$ in $n$ for $n>>0$. Since $n(\lambda,\mu)-m(\lambda,\mu) \geq 0$, each term in the above sum is a polynomial in $n$ of degree $\sum_{\kappa} \operatorname{dim} G/P_{\mu_\kappa}$ with non-negative leading term. We must therefore have $n(\lambda,\mu)= m(\lambda,\mu)$ for each $\lambda$. \end{proof}

\section{Topological descriptions}

\begin{sub}
	Recall that for $\lambda,\lambda' \in X(T)$ we write
	$$
	\lambda' \leq \lambda
	$$
	if $\lambda'_d +\ldots+\lambda_i' \leq \lambda_d +\ldots+\lambda_i$ for each $i$ with equality when $i=1$.
\end{sub}
\begin{proposition}\label{linearcomb}
	Let $\mu$ be a Hodge type with $\mu_\kappa-\rho$ dominant for each $\kappa$. Assume that
	$$
	\sum_{\kappa|_k = \kappa_0} \mu_{\kappa,1} - \mu_{\kappa,d} \leq e+p-1
	$$
	for each $\kappa_0:k \rightarrow \mathbb{F}$. Then:
	\begin{enumerate}
		\item There are integers $n(\lambda,\mu) \in \mathbb{Z}$ such that
		$$
		[M_\mu \otimes_{\mathcal{O}} \mathbb{F}] = \sum_\lambda n(\lambda,\mu) [M_{\widetilde{\lambda}} \otimes_{\mathcal{O}} \mathbb{F}]
		$$
		with the sum running over tuples $\lambda =(\lambda_{\kappa_0})$ with each $\lambda_{\kappa_0} \in X(T)$ dominant and satisfying $\lambda_{\kappa_0} \leq \sum_{\kappa|_{k} = \kappa_0}( \mu_\kappa-\rho)$.
		\item If $d=2$ then each $M_{\widetilde{\lambda}}\otimes_{\mathcal{O}} \mathbb{F}$ appearing in this sum is irreducible and generically reduced. Since the $[M_{\widetilde{\lambda}} \otimes_{\mathcal{O}} \mathbb{F}]$ are pairwise distinct this implies $n(\lambda,\mu) \geq 0$.
	\end{enumerate}
\end{proposition}

\begin{sub}
	To prove the proposition we will approximate $M_\mu \otimes_{\mathcal{O}} \mathbb{F}$ via explicit moduli conditions. In fact we give two such moduli interpretations, based on the following two operators:
	\begin{itemize}
		\item For any $\mathcal{O}$-algebra $A$ set
		$$
		\nabla:= u\frac{d}{du}: \mathfrak{S}_A[\frac{1}{E(u)}] \rightarrow \mathfrak{S}_A[\frac{1}{E(u)}]
		$$
		We also write $\nabla$ for the coordinate-wise extension to $\mathfrak{S}_A[\frac{1}{E(u)}]^d$.
		\item If $A$ is a $p$-adically complete $\mathcal{O}$-algebra of topologically finite type then for each $\sigma \in G_K$ we can also define
		$$
		\nabla_{\sigma}:= \frac{\sigma - \operatorname{Id}}{\mu} : A_{\operatorname{inf},A}[\frac{1}{\mu}] \rightarrow A_{\operatorname{inf},A}[\frac{1}{\mu}]
		$$
		Note this is well defined because $\sigma(\mu) \in \mu A_{\operatorname{inf}}$. We also note that $\frac{\mu}{\varphi^{-1}(\mu)`}$ and $E(u)$ generate the same ideal inside $A_{\operatorname{inf}}$ (as follows from Lemma~\ref{mu}) and so we can also view $\nabla_\sigma$ as an operator on $A_{\operatorname{inf},A}[\frac{1}{E(u)}]$. Again write $\nabla_{\sigma}$ also for the coordinate-wise extension to $A_{\operatorname{inf},A}[\frac{1}{E(u)}]^d$.
	\end{itemize}  
The advantage of $\nabla$ is that it is easier to compute with. The advantage of $\nabla_{\sigma}$ is that it is more directly related to Galois representations.
\end{sub}

\begin{lemma}
	There exist closed subfunctors $\operatorname{Gr}^{\nabla_{\sigma}}_{\mathcal{O}},\operatorname{Gr}^\nabla_{\mathcal{O}} \subset \operatorname{Gr}_{\mathcal{O}}$ such that
	\begin{enumerate}
		\item $\mathcal{E} \in \operatorname{Gr}^\nabla(A)$ if and only if
		$$
		E(u)\nabla(\mathcal{E}) \subset u\mathcal{E}
		$$
		as submodules of $\mathfrak{S}_A[\frac{1}{E(u)}]^d$.
		\item For any $p$-adically complete topologically finite type $\mathcal{O}$-algebra $A$, $\mathcal{E} \in \operatorname{Gr}^{\nabla_{\sigma}}(A)$ if and only if
		$$
		E(u)\nabla_{\sigma}(\mathcal{E}) \subset u\mathcal{E} \otimes_{\mathfrak{S}_A} A_{\operatorname{inf},A}
		$$
		as submodules of $A_{\operatorname{inf},A}[\frac{1}{E(u)}]^d$ for every $\sigma \in G_K$.
	\end{enumerate}
\end{lemma}
\begin{proof}
	That $\operatorname{Gr}^\nabla_{\mathcal{O}}$ is a closed subfunctor if clear, so we focus on $\operatorname{Gr}^{\nabla_{\sigma}}_{\mathcal{O}}$. Since $\operatorname{Gr}_{\mathcal{O}}$ is an inductive system of proper Noetherian $\mathcal{O}$-schemes it suffices to show that for any $\mathcal{E} \in \operatorname{Gr}_{\mathcal{O}}(A)$ the condition
	$$
	E(u)\nabla_{\sigma}(\mathcal{E}) \subset u\mathcal{E} \otimes_{\mathfrak{S}_A} A_{\operatorname{inf},A}
	$$
	is closed on $\operatorname{Spec}A$ whenever $A$ is a $p$-adically complete topologically finite type $\mathcal{O}$-algebra. This follows from an application of \cite[B.29]{EG19}.
	\end{proof}
\begin{remark}\label{explanation}
	Since $E(u)$ and $\frac{\mu}{\varphi^{-1}(\mu)}$ generate the same ideal in $A_{\operatorname{inf}}$ the condition defining $\operatorname{Gr}^{\nabla_{\sigma}}_{\mathcal{O}}$ can also be expressed as
	$$
	(\sigma-1)(\mathcal{E}) \subset u \varphi^{-1}(\mu)\mathcal{E} \otimes_{\mathfrak{S}_A} A_{\operatorname{inf},A}
	$$
	This description may be more familiar from the point of view of crystalline Breuil--Kisin modules. 
\end{remark}
\begin{proposition}\label{nablacontain}
	For every Hodge type $\mu$ one has $M_\mu \subset \operatorname{Gr}^{\nabla_{\sigma}}_{\mathcal{O}}$ and $M_\mu \subset \operatorname{Gr}^\nabla_{\mathcal{O}}$.
\end{proposition}
\begin{proof}
	Since $\operatorname{Gr}^\nabla_{\mathcal{O}}$ is closed it suffices to show $\mathcal{E} \in M_\mu(A)$ is contained in $\operatorname{Gr}^\nabla_{\mathcal{O}}$ for any $E$-algebra $A$. Lemma~\ref{points} allows us to write $\mathcal{E}$ as an intersection of $\mathcal{E}_\kappa = \frac{1}{E(u)^n} \sum E_{\kappa}(u)^{i +n} (A\otimes_{\mathbb{Z}_p} W(k))[u] \operatorname{Fil}^{-i}_\kappa$ for some $n\geq 0$. It is therefore enough to show that $E(u) \nabla(\mathcal{E}_\kappa) \subset u \mathcal{E}_\kappa$ and this follows since
	$$
	\nabla(\mathcal{E}_\kappa) = \sum_i \nabla(\frac{E_\kappa(u)^{i+n}}{E(u)^n})(A\otimes_{\mathbb{Z}_p} W(k))[u] \operatorname{Fil}^{-i}_\kappa
	$$ 
	and $E(u)\nabla(\frac{E_\kappa(u)^{i-n}}{E(u)^n}) \in u\frac{E_\kappa(u)^{i-n}}{E(u)^n} (A\otimes_{\mathbb{Z}_p} W(k))[u]$.
	
	There is a slight difficulty in giving an identical argument to show $M_\mu \subset \operatorname{Gr}^{\nabla_{\sigma}}_{\mathcal{O}}$ because the moduli description for $\operatorname{Gr}^{\nabla_{\sigma}}_{\mathcal{O}}$ does not apply when $A$ is an $E$-algebra. To address this we first note that the generic fibre of $M_\mu$ is reduced so to show $M_\mu \otimes_{\mathcal{O}} E \subset \operatorname{Gr}^{\nabla_{\sigma}}_{\mathcal{O}}$ it suffices to show this on $A$-points whenever $A$ is a finite extension of $E$. By the valuative criterion for properness, any such $A$-valued point is induced from a point valued in the ring of integers of $A$. Thus we are reduced to showing $M_\mu(A) \subset \operatorname{Gr}^{\nabla_{\sigma}}_{\mathcal{O}}(A)$ whenever $A$ is the ring of integers in a finite extension of $E$. Using part (3) of Lemma~\ref{points} this comes down to proving that 
	$$
	E(u)\nabla_{\sigma}(\mathcal{E}_\kappa) \subset u\mathcal{E}_\kappa \otimes_{\mathfrak{S}_A} A_{\operatorname{inf},A}
	$$
	for $\mathcal{E}_\kappa = \frac{1}{E(u)^n} \sum E_{\kappa}(u)^{i +n} \mathfrak{S}_A \operatorname{Fil}^{-i}_\kappa$. This would follow from the claim that
	$$
	E(u)\nabla_{\sigma}(\frac{E_\kappa(u)^{i+n}}{E(u)^n}) \in u\frac{E_\kappa(u)^{i+n}}{E(u)^{n}} A_{\operatorname{inf},\mathcal{O}}
	$$
To prove the claim first note that $\sigma(E_{\kappa}(u)) - E_{\kappa}(u) = \sigma(u) - u \in u \mu A_{\operatorname{inf}}$. Similarly $\sigma(E(u))-E(u) \in u\mu A_{\operatorname{inf}}$. Writing
$$
\nabla_{\sigma}(E_\kappa(u)^i) = \nabla_{\sigma}(E_\kappa(u)^{i-1})\sigma(E_\kappa(u)) + E_\kappa(u)\nabla_{\sigma}(E_\kappa(u))
$$
and arguing by induction on $i$ then gives that $\nabla_{\sigma}(E_\kappa(u)^i) \in u E_{\kappa(u)}^{i-1}A_{\operatorname{inf},\mathcal{O}}$. Similarly $\nabla_{\sigma}(E(u)^i) \in u E(u)^{i-1}A_{\operatorname{inf}}$. Since we can write
$$
E(u)^n (\sigma-1)(\frac{E_{\kappa}(u)^{i+n}}{E(u)^n}) = (\sigma-1)(E_{\kappa}(u)^{i+n}) - \frac{\sigma(E_{\kappa}(u)^{i+n})}{\sigma(E(u)^n)} (\sigma-1)(E(u)^n)
$$
the claim follows.
\end{proof}
\begin{remark}
	After possibly replacing the compatible system of primitive $p$-th power roots of unity $\epsilon$ we can choose $\sigma \in G_K$ so that $\sigma(u)/u = \epsilon$. Then
	$$
	\nabla_\sigma(u^i) = u^i \left(  \frac{\frac{\sigma(u)}{u} - 1}{[\epsilon] -1} \right) = u^i \left( \frac{[\epsilon^{i}] - 1}{[\epsilon]-1}\right) = u^i(1+ [\epsilon] + \ldots + [\epsilon]^{i-1})
	$$
	Thus $\nabla_\sigma = u \nabla_q$ where $\nabla_q$ is the $q$-derivation for $q = [\epsilon]$. In particular $\nabla_\sigma \equiv u\frac{d}{du} = \nabla$ modulo $[\epsilon]-1$. This illustrates the close relationship between the $\operatorname{Gr}^{\nabla_{\sigma}}_{\mathcal{O}}$ and the locus $\operatorname{Gr}^\nabla_{\mathcal{O}}$.
\end{remark}
\begin{sub}
For the rest of this section we focus on $\operatorname{Gr}^{\nabla}_{\mathcal{O}} \otimes_{\mathcal{O}} \mathbb{F}$. Note that since $\nabla$ is $W(k)$-linear we have
$$
\operatorname{Gr}^{\nabla}_{\mathcal{O}} = \prod_{\kappa_0} \operatorname{Gr}_{\kappa_0}^\nabla
$$ 
where $\operatorname{Gr}_{\kappa_0}^\nabla$ is defined similarly. Let us write $\overline{\operatorname{Gr}} = \operatorname{Gr}_{\kappa_0} \otimes_{\mathcal{O}} \mathbb{F}$ (note this is independent of $\kappa_0$) and $\overline{\operatorname{Gr}}^\nabla = \operatorname{Gr}^\nabla_{\kappa_0} \otimes_{\mathcal{O}} \mathbb{F}$. The description from Lemma~\ref{beueville} shows that the group scheme
$$
LG^+: A\mapsto \operatorname{GL}_d(A[[u]])
$$ 
acts on $\overline{\operatorname{Gr}}$. For $\lambda \in X(T)$ dominant we set $\overline{\operatorname{Gr}}_{\lambda}$ equal to the $LG^+$-orbit of $\mathcal{E}_{\lambda} \in \operatorname{Gr}$ (recall $\mathcal{E}_\lambda$ is defined in \ref{flags}) and we set $\overline{\operatorname{Gr}}_{\leq \lambda}$ equal to its reduced closure. Then
$$
\overline{\operatorname{Gr}}_{\leq \lambda} = \bigcup_{\lambda' \leq \lambda} \overline{\operatorname{Gr}}_{\lambda'}
$$

\end{sub}
\begin{lemma}\label{dimC}
	Suppose $\lambda \in X(T)$ is dominant with
	$$
	\lambda_{1} - \lambda_{d} \leq e+ p-1
	$$
	Set $\mathcal{C}_\lambda$ equal to the closure of $\overline{\operatorname{Gr}}_{\lambda} \cap \overline{\operatorname{Gr}}^{\nabla}$ in $\overline{\operatorname{Gr}}$. Then $\mathcal{C}_{\lambda}$ is reduced and irreducible of dimension
	$$
	\sum_{\kappa_0} \sum_{i <j} \operatorname{max}\lbrace\lambda_i-\lambda_j,e\rbrace
	$$
\end{lemma}
\begin{proof}
	We begin by giving an open cover of $\overline{\operatorname{Gr}}_\lambda$: let $\mathcal{U}_\lambda \subset L^+G$ denote the subfunctor whose $A$-points consist of unipotent upper triangular matrices 	
	$$
	\begin{pmatrix}
		1 & & a_{ij} \\ & \ddots & \\ & & 1
	\end{pmatrix} \in L^+G(A)
	$$ 
	where for each $i> j$, $a_{ij} \in A[u]$ has degree $< \lambda_j - \lambda_i$. Consider the morphism $\mathcal{U}_\lambda \rightarrow \overline{\operatorname{Gr}}_\lambda$ sending $g \mapsto g\mathcal{E}_\lambda$. Recall that $g\mathcal{E}_\lambda \mapsto g_0\mathcal{E}_{\lambda}$, for $g_0 = g$ modulo $u$, defines a morphism $\overline{\operatorname{Gr}}_{\lambda} \rightarrow G/P_{\lambda}$. Since the parabolic $P_{\lambda}$ is contained in the Borel of lower triangular matrices $B^- \subset G$ we can compose this map with $G/P_{\lambda} \rightarrow G/B^{-}$. Then the morphism $\mathcal{U} \rightarrow \overline{\operatorname{Gr}}_{\lambda}$ identities $\mathcal{U}_{\lambda}$ with the preimage under this composite of the open $U \subset G/B^-$ consisting of upper triangular unipotent matrices. In particular, $\mathcal{U}_{\lambda} \rightarrow \overline{\operatorname{Gr}}_{\lambda}$ is an open immersion and $\overline{\operatorname{Gr}}_{\lambda} = \bigcup_{w } w\mathcal{U}_{\lambda}$ with $w$ running over the permutation matrices in $G$ (as follows by considering the open cover $G/B^- = \bigcup w U$). 
	
	Since $\nabla(w) = 0$ we have $w \mathcal{U}_{\lambda} \cap\overline{\operatorname{Gr}}^\nabla = w(\mathcal{U}_{\lambda} \cap \overline{\operatorname{Gr}}^{\nabla})$. Therefore the lemma reduces to showing $\mathcal{U}_{\lambda} \cap\overline{ \operatorname{Gr}}^\nabla$ is an affine space of the claimed dimension. Observe that $g \in \mathcal{U}_{\lambda} \cap \operatorname{Gr}^\nabla$ if and only if 
	\begin{equation}\label{nabla}
		u^{e-1} g^{-1} \nabla(g) \in \left(\begin{smallmatrix}
			u^{\lambda_1} & & \\ & \ddots & \\ & & u^{\lambda_d}
		\end{smallmatrix}\right) \operatorname{Mat}(A[u])  \left(\begin{smallmatrix}
			u^{-\lambda_1} & & \\ & \ddots & \\ & & u^{-\lambda_d}
		\end{smallmatrix}\right)
	\end{equation}
	If we write $g^{-1} = (b_{ij})_{ij}$ then, using that $b_{jj} = 1$, $b_{lj} =0$ for $l<j$, and $\nabla(a_{ii}) =0$, we see that \eqref{nabla} is equivalent to asking that
	\begin{equation}\label{nabla2}
		\nabla(a_{ij}) +\sum_{j<l<i} \nabla(a_{il})  b_{lj} \in u^{\lambda_j-\lambda_i-e+1}A[u] 
	\end{equation}
	for every $i>j$. By assumption $\lambda_j -\lambda_i -e+1 \leq p$ and so $\sum_{j<l<i} \nabla(a_{il})  b_{lj}$ modulo $u^{\lambda_j-\lambda_i -e+1}$ admits an antiderivative; in other words, there exists a unique $X \in uA[u]$ of degree $< \lambda_j -\lambda_i -e+1$ with
	$$
	\nabla (X) \equiv -\sum_{j<l<i} \nabla(a_{il})  b_{lj} \quad \text{ modulo } u^{\lambda_j-\lambda_i -e+1}
	$$
	Since $a_{ij}$ has degree $< \lambda_j - \lambda_i$ it follows that 
	$$
	a_{ij}  = \begin{cases}
		X +  a_{ij}^{(0)} + u^{\lambda_j-\lambda_i  -e+1 }a_{ij}^{(1)}+ \ldots +  u^{\lambda_j - \lambda_i -1 }a_{ij}^{(e-1)} & \text{if $\lambda_j -\lambda_i \geq e$ } \\
		X+ a_{ij}^{(0)} + u a_{ij}^{(1)} + \ldots + u^{\lambda_j -\lambda_i -1} a_{ij}^{\lambda_i-\lambda_j-1} & \text{if $\lambda_j -\lambda_i < e$}
	\end{cases}
	$$
	for some $a_{ij}^{(l)} \in A$. Note that, for $i > j$, the $ij$-th entry of $gg^{-1} = 1$ is
	$$
	0 = \sum_{l=0}^d a_{il} b_{lj} = b_{ij} + a_{ij}+ \sum_{j<l<i} a_{il}b_{lj}
	$$
	This shows, by an inductive argument, that $b_{ij}$ is a function of $a_{lk}$ for $l<k$ with $k-l \leq i-j$. Therefore the element $X \in uA[u]$ considered above depends on $a_{lk}$ with $k-l < i-j$. As a consequence the morphism
	$$
	\mathcal{U}_{\lambda} \cap \overline{\operatorname{Gr}}^\nabla \rightarrow \prod_{ij} \mathbb{A}_{\mathbb{F}}^{\operatorname{min}\lbrace e,\lambda_j-\lambda_i\rbrace}
	$$
	given by $(a_{ij}) \mapsto ( a_{ij}^{(l)} )$ has a well-defined inverse which finishes the proof.
\end{proof}
\begin{proof}[Proof of Proposition]
	First observe that under the identification $\operatorname{Gr}_{\mathcal{O}} \otimes_{\mathcal{O}} \mathbb{F} \cong \prod_{\kappa_0} \overline{\operatorname{Gr}}$
	we have $(M_\mu \otimes_{\mathcal{O}} \mathbb{F})_{\operatorname{red}} \hookrightarrow \prod_{\kappa_0} \overline{\operatorname{Gr}}_{\leq \sum_{\kappa|_{k} = \kappa_0} \mu_\kappa}$. Thus $(M_\mu \otimes_{\mathcal{O}} \mathbb{F})_{\operatorname{red}}$ is contained in $\bigcup_{\lambda =(\lambda_{\kappa_0})} \prod \overline{\operatorname{Gr}}_{\lambda}$ where the product runs over $\lambda = (\lambda_{\kappa_0})$ with $\lambda_{\kappa_0} \leq \sum_{\kappa|_{k} = \kappa_0} \mu_\kappa$. Since $M_\mu \subset \operatorname{Gr}^\nabla_{\mathcal{O}}$ and each $\mu_\kappa-\rho$ is dominant the dimension calculations from Lemma~\ref{dimC} imply that
	$$
	(M_\mu \otimes_{\mathcal{O}} \mathbb{F})_{\operatorname{red}} \subset \bigcup_{\lambda = (\lambda_{\kappa_0})} \mathcal{C}_{\lambda +e\rho}
	$$
	where the union now runs over $\lambda =(\lambda_{\kappa_0})$ with $\lambda_{\kappa_0} +e\rho\leq \sum_{\kappa|_{k} = \kappa_0} \mu_\kappa $ and where we write $\mathcal{C}_{\lambda+e\rho} = \prod_{\kappa_0} \mathcal{C}_{\lambda_{\kappa_0}+e\rho}$. Thus, one can write 
	$$
	[M_\mu \otimes_{\mathcal{O}} \mathbb{F}] = \sum_{\lambda} n(\lambda,\mu) [\mathcal{C}_{\lambda +e\rho}]
	$$
	as cycles, for integers $n(\lambda,\mu) \geq 0$. Furthermore, since $\overline{\operatorname{Gr}}_{\sum_{\kappa|_{k} = \kappa_0} \mu_\kappa} \subset \overline{\operatorname{Gr}}_{\leq \sum_{\kappa|_{k} = \kappa_0} \mu_\kappa}$ is open it follows that $\mathcal{C}_{\lambda+e\rho} \cap (M_\mu \otimes_{\mathcal{O}} \mathbb{F})$ is open in $(M_\mu \otimes_{\mathcal{O}} \mathbb{F})$ for $\lambda = (\lambda_{\kappa_0})$ with $\lambda_{\kappa_0} = \sum_{\kappa|_k = \kappa_0} (\mu_{\kappa} -\rho)$. Since this intersection is clearly non-empty it follows that $n(\lambda,\mu) =1$ for this particular $\lambda$.
	
	This shows that
	$$
	[M_{\widetilde{\lambda}} \otimes_{\mathcal{O}} \mathbb{F}] - [\mathcal{C}_{\lambda + e\rho}]
	$$
	can be expressed as a $\mathbb{Z}_{\geq 0}$-linear combination of $[\mathcal{C}_{\lambda'+e\rho}]$'s for $\lambda' = (\lambda'_{\kappa_0})$ with $\lambda_{\kappa_0}' \leq \lambda_{\kappa_0}$. Arguing by induction we conclude that we can always write $[\mathcal{C}_{\lambda+e\rho}] = \sum_{\lambda'} n_{\lambda'}[M_{\widetilde{\lambda}'} \otimes_{\mathcal{O}} \mathbb{F}]$ for $\lambda' = (\lambda'_{\kappa_0})$ satisfying $\lambda'_{\kappa_0} \leq \lambda_{\kappa_0}$ and some $n_{\lambda'} \in \mathbb{Z}$. This proves the first part of the proposition.
	
	The second part follows from the first provided we can show $M_{\widetilde{\lambda}}$ is irreducible whenever $\lambda_{\kappa_0,1}-\lambda_{\kappa_0,d} \leq p-1$. To establish this irreducibility we require $d=2$. Choose an indexing $\kappa_{0,1},\ldots,\kappa_{0,e}$ of those $\kappa$ with $\kappa|_k =\kappa_0$ so that $\kappa_{0,1} = \widetilde{\kappa}_0$. Then construct a scheme $X$ which classifies tuples $(\mathcal{E}_e \subset \ldots \subset \mathcal{E}_1)$ with $\mathcal{E}_1 \in \prod_{\kappa_0} (G/P_{\kappa_{0,1}} \otimes_{\mathcal{O}_K,\kappa_{0,1}} \mathcal{O})$ and
	$$
	\left( \prod_{\kappa_0} E_{\kappa_{0,i}}(u)\right) \mathcal{E}_i \subset \mathcal{E}_{i+1}\subset \mathcal{E}_i
	$$
	with $\mathcal{E}_i/\mathcal{E}_{i+1}$ of rank one over $(A\otimes_{\mathbb{Z}_p} W(k))$ for each $i$. Then the map $(\mathcal{E}_i) \mapsto \mathcal{E}_e$ produces a proper morphism $X \rightarrow \operatorname{Gr}_{\mathcal{O}}$ which on the generic fibre identifies $X \otimes_{\mathcal{O}} E$ with $M_{\widetilde{\lambda}} \otimes_{\mathcal{O}} E$. In particular, this shows that $X \otimes_{\mathcal{O}} \mathbb{F} \rightarrow M_{\widetilde{\lambda}} \otimes_{\mathcal{O}} \mathbb{F}$ is surjective. On the other hand, $X$ is a successive extension of (products of) grassmannians over a (product of) flag varieties. Thus $X$ is $\mathcal{O}$-smooth, and so $X \otimes_{\mathcal{O}} \mathbb{F}$ is irreducible. We conclude the same is true of $M_{\widetilde{\lambda}} \otimes_{\mathcal{O}} \mathbb{F}$. 
\end{proof}
\section{Duality}\label{dual}

In this section we introduce an involution of $\operatorname{Gr}_{\mathcal{O}}$ which is useful when dealing with certain normalisation issues which arise when passing between the affine grassmannian and moduli of Breuil--Kisin modules.
\begin{sub}
	If $\mathcal{E}$ corresponds to an $A$-valued point of $\operatorname{Gr}_{\mathcal{O}}$ then 
	$$
	\mathcal{E}^* := \operatorname{Hom}_{(A \otimes_{\mathbb{Z}_p} W(k))[u]}(\mathcal{E}, (A\otimes_{\mathbb{Z}_p} W(k))[u])
	$$
	is again $(A\otimes_{\mathbb{Z}_p} W(k))[u]$-projective and, under the natural identification $(A\otimes_{\mathbb{Z}_p} W(k))[u]^{d,*} \cong  (A\otimes_{\mathbb{Z}_p} W(k))[u]^{d}$, we can view $\mathcal{E}^*$ as an $A$-valued point of $\operatorname{Gr}_{\mathcal{O}}$. Since $\mathcal{E}^{**} = \mathcal{E}$ the endomorphism of $\operatorname{Gr}_{\mathcal{O}}$ induced by
	$$
	\mathcal{E} \mapsto \mathcal{E}^*
	$$
	is an automorphism.
\end{sub}

\begin{lemma}\label{dual2}
	The above automorphism identifies $M_{\lambda}$ with $M_{-w_0\lambda}$ where $w_0 \in W$ denotes the longest element. 
\end{lemma}
In other words, $-w_0\mu = (-w_0\mu_\kappa)$ where $-w_0\mu_\kappa = (-\mu_{\kappa,d},\ldots,-\mu_{\kappa,1}) \in X(T)$.
\begin{proof}
	It suffices to prove this on the generic fibre. Thus, one is reduced to prove that for any $\lambda \in X(T)$, $G/P_{\lambda} \subset \operatorname{Gr}$ is identified with $G/P_{-w_0\lambda}$ by the version of $\mathcal{E} \mapsto \mathcal{E}^*$ on $\operatorname{Gr}$. But this follows easily from the fact that if $\mathcal{E}$ is generated by $(e_1,\ldots,e_d)X$ then $\mathcal{E}^*$ is generated by $(e_1,\ldots,e_d)(X^{-1})^t$ for $(X^{-1})^{t}$ the conjugate transpose. In particular, the $G$-orbit of any $\mathcal{E}$ is mapped onto the $G$-orbit of $\mathcal{E}^*$. Since $\mathcal{E}_{\lambda}^* = w_0 \mathcal{E}_{-w_0\lambda}$ the lemma follows.
\end{proof}
\section{Breuil--Kisin modules}\label{BKm}

	\begin{sub}
		Let $A$ be a $p$-adically complete $\mathcal{O}$-algebra. Then a \emph{Breuil--Kisin module} $\mathfrak{M}$ over $A$ is a finite projective $\mathfrak{S}_A$-module equipped with an $\mathfrak{S}_A$-linear homomorphism
	$$
	\varphi_{\mathfrak{M}}= \varphi :\mathfrak{M} \otimes_{\varphi,\mathfrak{S}_A} \mathfrak{S}_A \rightarrow \mathfrak{M}
	$$
	whose cokernel is killed by a power of $E(u)$. We say $\mathfrak{M}$ has height $\leq h$ if the cokernel is killed by $E(u)^{h}$. We write $\mathfrak{M}^\varphi$ for the image of $\varphi_{\mathfrak{M}}$ and $\varphi(\mathfrak{M})$ for the image of the composite $\mathfrak{M} \xrightarrow{m\mapsto m \otimes 1} \mathfrak{M} \otimes_{\varphi,\mathfrak{S}_A} \mathfrak{S}_A \rightarrow \mathfrak{S}$. Thus $\varphi(\mathfrak{M})$ is an $\varphi(\mathfrak{S}_A)$-submodule of $\mathfrak{M}^\varphi$ which generates $\mathfrak{M}^\varphi$ over $\mathfrak{S}_A$. 
	\end{sub}
	\begin{definition}
		For any $p$-adically complete $\mathcal{O}$-algebra $A$ write 
		\begin{itemize}
			\item $Z^{\leq h}_d(A)$ for the groupoid of rank $d$ Breuil--Kisin modules over $A$ with height $\leq h$. Morphisms are $\mathfrak{S}_A$-linear isomorphisms compatible with the Frobenius.
			\item $\widetilde{Z}^{\leq h}_d(A)$ for the groupoid of pairs $(\mathfrak{M},\beta)$ with $\mathfrak{M} \in Z^{\leq h}_d(A)$ and $\beta= (\beta_1,\ldots,\beta_d)$ an $\mathfrak{S}_A$-basis of $\mathfrak{M}$. Morphisms are $\mathfrak{S}_A$-linear isomorphisms compatible with the Frobenius and identifying the bases.
		\end{itemize}
	 With pull-backs defined by base-change these categories form an fpqc stacks $Z^{\leq h}, \widetilde{Z}^{\leq h}_d$ over $\operatorname{Spf}\mathcal{O}$.
	\end{definition}
	
	\begin{sub}\label{GammaandPsiconstruction}
		We consider the following diagram:
		\begin{equation*}\label{PRconst}
			\begin{tikzcd}
				& \ar[dl,"\Gamma"] \ar[dr,"\Psi"] \widetilde{Z}^{\leq h}_d & \\
					Z^{\leq h}_d  & &  \operatorname{Gr}_{\mathcal{O}}
			\end{tikzcd}
		\end{equation*}
		where $\Gamma$ forgets the choice of basis and, with $\operatorname{Gr}_{\mathcal{O}}$ viewed as a formal scheme over $\operatorname{Spf}\mathcal{O}$, the morphism $\Psi$ sends $(\mathfrak{M},\beta)$ onto $\mathcal{E} \in \operatorname{Gr}_{\mathcal{O}}$ obtained via
		$$
		\mathcal{E} := \mathfrak{M} \hookrightarrow \frac{1}{E(u)^h}\mathfrak{M}^\varphi \xrightarrow{\varphi(\beta)} \frac{1}{E(u)^h}\mathfrak{S}_A^d
		$$
		(here $\varphi(\beta)$ is interpreted as an identification $\mathfrak{M}^\varphi \cong \mathfrak{S}_A^d$). Concretely, if $C$ is the matrix with entries in $\mathfrak{S}_A$ for which $\varphi(\beta) = \beta C$ then, since $\mathfrak{M}$ is generated by $\varphi(\beta) C^{-1}$, we have $\Psi(\mathfrak{M},\beta) = \mathcal{E}$ for $\mathcal{E} \subset \mathfrak{S}_A[\frac{1}{E(u)}]^d$ the submodule generated by $C^{-1}$. 
		
		Writing $L^+\operatorname{GL}_d$ for the group scheme given by $A \mapsto \operatorname{GL}_d(\mathfrak{S}_A)$ we see that $\Gamma$ is an $L^+\operatorname{GL}_d$-torsor for the action on $\widetilde{Z}^{\leq h}_d$ given by $g \cdot (\mathfrak{M},\beta) = (\mathfrak{M},\beta g)$. A second action of $L^+\operatorname{GL}_d$ on $\widetilde{Z}^{\leq h}$ can be given by 
		$$
		g* (\mathfrak{M},\beta)  = (\mathfrak{M}_g,\beta)
		$$
		where $\mathfrak{M}_g = \mathfrak{M}$ as an $\mathfrak{S}_A$-module, with Frobenius given by $\varphi_g(\beta) = \beta gC$ for $C$ the matrix determined by $\varphi(\beta) = \beta C$. It is easy to see that this action makes $\Psi$ into an $L^+\operatorname{GL}_d$-torsor over its image in $\operatorname{Gr}_{\mathcal{O}}$.
	\end{sub}
\begin{sub}\label{GammaandPhivariant}
	An alternative viewpoint on~\ref{GammaandPsiconstruction} is as follows. Let $L^{\leq h}\operatorname{GL}_d$ denote the group scheme over $\mathcal{O}$ given by
	$$
	A \mapsto \lbrace C \in \operatorname{GL}_d(\mathfrak{S}_A[\tfrac{1}{E(u)}] \mid  E(u)^h C^{-1}, C \in \operatorname{Mat}(\mathfrak{S}_A) \rbrace
	$$
	Then the morphism $\widetilde{Z}^{\leq h}_d \rightarrow L^{\leq h}\operatorname{GL}_d$ given by $(\mathfrak{M},\beta) \mapsto C$, for $C$ defined by $\varphi(\beta) = \beta C$, is an isomorphism (here we view $L^\leq h \operatorname{GL}_d$ as a formal scheme over $\operatorname{Spf}\mathcal{O}$). Under this isomorphism the action $g \cdot (\mathfrak{M},\beta)$ corresponds to (right) action of $L^+\operatorname{GL}_d$ by $\varphi$-conjugation: $C \mapsto g^{-1} C \varphi(g)$. The action $g * (\mathfrak{M},\beta)$ corresponds to the (left) multiplication action: $C \mapsto gC$. Thus~\eqref{PRconst} identifies with the diagram
	\begin{equation*}
	\begin{tikzcd}
		& \ar[dl] \ar[dr,"C \mapsto C^{-1}"] L^{\leq h} \operatorname{GL}_d & \\
		\left[  L^{\leq h} \operatorname{GL}_d / \prescript{}{\varphi}{L^+\operatorname{GL}_d} \right]   & & \left[ L^{\leq h}\operatorname{GL}_d / L^+\operatorname{GL}_d \right] \hookrightarrow \operatorname{Gr}_{\mathcal{O}}
	\end{tikzcd}
\end{equation*}
Here $L^{\leq h} \operatorname{GL}_d / \prescript{}{\varphi}{L^+\operatorname{GL}_d}$ indicates the quotient by Frobenius conjugation
\end{sub}

The issue with the construction in~\ref{GammaandPsiconstruction} is that $\widetilde{Z}^{\leq h}_d \cong L^{\leq h}\operatorname{GL}_d$ is not  of finite type over $\operatorname{Spf}\mathcal{O}$. To address this we instead consider:
\begin{definition}
	For $N \geq 1$ let $\widetilde{Z}^{\leq h,N}_d(A)$ denote the category of pairs $(\mathfrak{M},\beta) \in \widetilde{Z}_d^{\leq h}(A)$ whose morphisms are $\varphi$-equivariant morphisms of $\mathfrak{S}_A$-modules which identify the bases modulo $u^N$. Equivalently, $\widetilde{Z}^{\leq h,N}_d$ is the quotient of $\widetilde{Z}^{\leq h}_d$ under the action of the group scheme 
	$$
	U_N : A \mapsto 1 + u^N\operatorname{Mat}(\mathfrak{S}_A)
	$$ 
	Thus $\widetilde{Z}^{\leq h}_d \cong [L^{\leq h}\operatorname{GL}_d/U_N]$ (the quotient being by Frobenius conjugation). 
\end{definition}
The following is a essentially \cite[2.1]{PR09}.
\begin{proposition}\label{basiclocalmodel}
 Fix $n \geq 1$. Then, for $N$ sufficiently large, \eqref{GammaandPsiconstruction} induces a diagram
 \begin{equation*}\label{PRconst}
 	\begin{tikzcd}
 		& \ar[dl,"\Gamma_N"] \ar[dr,"\Psi_N"] \widetilde{Z}^{\leq h,N}_d  \otimes_{\mathcal{O}} \mathcal{O}/\pi^n& \\
 		Z^{\leq h}_d \otimes_{\mathcal{O}} \mathcal{O}/\pi^n & &  \operatorname{Gr}_{\mathcal{O}} \otimes_{\mathcal{O}} \mathcal{O}/\pi^n
 	\end{tikzcd}
 \end{equation*}
 in which $\Gamma_N$ is a $\mathcal{G}_N$-torsor for $\mathcal{G}_N$ the group scheme defined by $A\mapsto \operatorname{GL}_d(\mathfrak{S}_A/u^N)$ and $\Psi_N$ is smooth of relative dimension $\operatorname{dim}_{\mathcal{O}}\mathcal{G}_N$ with irreducible fibres. More precisely, the action $g * (\mathfrak{M},\beta)$ from~\ref{GammaandPsiconstruction} induces an action of $\mathcal{G}_N$ on $\widetilde{Z}^{\leq h,N}_{d} \otimes_{\mathcal{O}} \mathcal{O}/\pi^n$ making $\Psi_N$ into a $\mathcal{G}_N$-torsor over its image in $\operatorname{Gr}_{\mathcal{O}} \otimes_{\mathcal{O}} \mathcal{O}/\pi^n$.
\end{proposition}

 We will see in the proof that $N$ is required large enough that $E(u)^h$ divides $u^{(p-1)N-1}$ in $\mathfrak{S}_{\mathcal{O}/\pi^n}$ (this can be made explicit using e.g. \cite[5.2.6]{EG21}).
\begin{proof}
	The crucial observation is that if $E(u)^h$ divides $u^{(p-1)N-1}$ in $\mathfrak{S}_{\mathcal{O}/\pi^n}$ then, as stacks over $\operatorname{Spec}\mathcal{O}/\pi^n$, the identity map $L^{\leq h}\operatorname{GL}_d \rightarrow L^{\leq h}\operatorname{GL}_d$ induces an isomorphism
	$$
	\left[ L^{\leq h}\operatorname{GL}_d / \prescript{}{\varphi}{U_N} \right] \cong [U_N\backslash L^{\leq h}\operatorname{GL}_d]
	$$
	(the quotient on the left being by Frobenius conjugation and that on the right being by left multiplication). This immediately shows that the morphism $\Psi$ induces a morphism $\Psi_N$ as claimed, and that $\Psi_N$ is a $\mathcal{G}_N = L^+\operatorname{GL}_d/U_N$-torsor over the closed $\mathcal{O}/\pi^a$-subscheme $[L^{\leq h}\operatorname{GL}_d/L^+\operatorname{GL}_d] \hookrightarrow \operatorname{Gr}_{\mathcal{O}}$. 
	
	Concretely, the claimed isomorphism follows from the two assertions: If $C \in L^{\leq h}\operatorname{GL}_d(A)$
	\begin{itemize}
		\item For each $g_0 \in U_N(A)$ we have $g_0^{-1}C \varphi(g_0) = gC$ for a unique $g \in U_N(A)$.
		\item For each $g \in U_N(A)$ there exists a unique $g_0 \in U_N(A)$ with $g_0^{-1} C\varphi(g_0) = gC$.
	\end{itemize}
	Note these are exactly the same assertions as (1) and (2) in \cite[2.2]{PR09}. The first point is easy because $g_0 \in U_N(A)$ implies
	$$
	g = g_0^{-1}C\varphi(g_0)C^{-1} = g_0^{-1} \left( 1 + u^{pN}C_1 C^{-1} \right) = 1 + u^N C_2 + u^{pN}C_3 C^{-1}
	$$
	Therefore, the fact that $E(u)^hC^{-1} \in \operatorname{Mat} (\mathfrak{S}_A)$ and that $E(u)^h$ divides $u^{pN}$ ensures $g \in U_N(A)$. For the second point notice that, if $J_{n} = (gC)\varphi(gC)\ldots\varphi^{n}(gC)$ and $I_{n} = C\varphi(C)\ldots\varphi^{n-1}(C)$, then any such $g_0$ satisfies
	$$
	g_0 = I_n \varphi^n(g_0) J_n^{-1} = I_n J_{n}^{-1} + I_n( \varphi^n(g_0) -1) J_n^{-1}
	$$
	for every $n \geq 1$. We claim that for any $g_1 \in U_N(A)$ we have $I_n( \varphi^n(g_1) -1) J_n^{-1} \in u^N\operatorname{Mat}(\mathfrak{S}_A)$ and that this element conveges $u$-adically to zero as $n \rightarrow \infty$. This ensures that $I_n J_n^{-1}$ converges in $U_N(A)$ as $n \rightarrow \infty$ and that this limit is the unique $g_0$ satisfying $g_0^{-1} C \varphi(g_0) = gC$. Since $g_1\equiv 1$ modulo $u^N$ we can write $\varphi^n(g_1)-1 = u^{N + (p^n-1)N}g_1'$. Therefore the claim will follow if  $u^{(p^n-1)N}J_{n-1}^{-1} \in \operatorname{Mat}(\mathfrak{S}_A)$ converges $u$-adically to zero. Since $\varphi^i(E(u)^h)$ divides $u^{((p-1)N-1)p^i}$ and 
	$$
	E(u)^{h} \varphi(E(u))^{h} \ldots \varphi^{n-1}(E(u)^{h}) I_{n-1}^{-1} \in \operatorname{Mat}(\mathfrak{S}_A)
	$$
	the claim follows from the observation that $(p^n-1)N -((p-1)N-1)(1+\ldots+p^{n-1}) = 1+\ldots+ p^{n-1}$ is $\geq 0$ and $\rightarrow \infty$.
\end{proof}
\begin{corollary}
	$\widetilde{Z}^{\leq h,N}_d \times_{\mathcal{O}} \mathcal{O}/\pi^n$ is a finite type $\mathcal{O}$-scheme for $N>>0$ and $Z^{\leq h}_d$ is a $p$-adic formal algebraic stack (in the sense of \cite[A7]{EG19}) of finite type over $\operatorname{Spf}\mathcal{O}$.
\end{corollary}
\begin{proof}
	The first part follows since we've just seen that  $\widetilde{Z}^{\leq h,N}_d \times_{\mathcal{O}} \mathcal{O}/\pi^n$ is a torsor for a finite type group scheme over a finite type $\mathcal{O}/\pi^n$-scheme. The second part follows from the first and the definition of a $p$-adic formal algebraic stack.
\end{proof}

\section{Crystalline Breuil--Kisin modules}

\begin{sub}
	If $A$ is a $p$-adically complete $\mathcal{O}$-algebra which is of topologically finite type then a \emph{crystalline Breuil--Kisin module over $A$} is a pair $(\mathfrak{M},\sigma)$ with $\mathfrak{M}$ a Breuil--Kisin module over $A$ and $\sigma$ a continuous $\varphi$-equivariant $A_{\operatorname{inf},A}$-semilinear action of $G_K$ on $\mathfrak{M} \otimes_{\mathfrak{S}_A} A_{\operatorname{inf},A}$ satisfying
$$
(\sigma-1)(m) \in \mathfrak{M} \otimes_{\mathfrak{S}_A} [\pi^\flat]\varphi^{-1}(\mu)A_{\operatorname{inf},A}, \qquad (\sigma_{\infty} -1)(m) = 0
$$
for every $m \in \mathfrak{M}$ and every $\sigma \in G_K, \sigma_{\infty} \in G_{K_\infty}$.
\end{sub}
\begin{definition}
	Write $Y^{\leq h}_d(A)$ for the groupoid consisting of rank $d$ crystalline Breuil--Kisin modules over $A$ with height $\leq h$.
\end{definition}

\begin{sub}
	One can attach a Hodge type to crystalline Breuil--Kisin modules (at least over coefficient rings which are $\mathcal{O}$-flat): For $n \in \mathbb{Z}$ one defines $\operatorname{Fil}^n(\mathfrak{M}^\varphi) = \mathfrak{M}^\varphi \cap E(u)^n\mathfrak{M}$ and equips the finite projective $\mathcal{O}_K\otimes_{\mathbb{Z}_p} A$-module $\mathfrak{M}^\varphi/E(u)$ with the filtration whose $n$-th filtered piece is the image of $\operatorname{Fil}^n(\mathfrak{M}^\varphi)$. The graded pieces of this filtration become $(A \otimes_{\mathbb{Z}_p} W(k)[\frac{1}{p}]$-projective after inverting $p$. This allows us to say that $(\mathfrak{M},\sigma)$ has Hodge type $\mu = (\mu_{\kappa})$ if the part of $\operatorname{gr}^n\mathfrak{M}^\varphi[\frac{1}{p}]/E(u)$ on which $K$ acts via $\kappa$ has $E$-dimension equal the multiplicity of $n$ in $\mu_\kappa$.
\end{sub}
\begin{remark}
	In other words, $\mathfrak{M}$ has Hodge type $\mu$ if the part of $\mathfrak{M}^\varphi[\frac{1}{p}]/E(u)$ on which $K$ acts via $\kappa$ is a filtration of type 
	$$
	-w_0\mu_\kappa = (-\mu_{\kappa,d},\ldots,-\mu_{\kappa,1})
	$$
	in the sense of \ref{flags}.
\end{remark}
 
 \begin{theorem}\label{crys}
 	If $(\mathfrak{M},\sigma) \in Y^{\leq h}_d(A)$ with $A$ a finite flat $\mathcal{O}$-algebra then
 	$$
 	V = (\mathfrak{M} \otimes_{\mathfrak{S}_A} W(C^\flat)_A)^{\varphi=1}
 	$$
 	equipped with the $G_K$-action induced by $\sigma$ is a crystalline representation of $G_K$ on a finite projective $A$-module. Furthermore, the Hodge type of $(\mathfrak{M},\sigma)$ coincides with that attached to $V$ via the filtered module $D_{\operatorname{crys}}(V)_K := (V \otimes_{\mathbb{Z}_p} B_{\operatorname{dR}})^{G_K}$ with $n$-th filtered piece given by $(V\otimes_{\mathbb{Z}_p} t^nB_{\operatorname{dR}}^+)^{G_K}$.
 \end{theorem}
\begin{proof}
	The theorem as stated is taken from \cite[2.1.12]{B19}, but the result originates from a combination of ideas appearing in \cite{Kis06,GLS,Ozeki}.
\end{proof}
\begin{remark}
	These conventions mean that the Hodge type of the cyclotomic character is $-1$.
\end{remark}
\begin{proposition}
	There exists a limit preserving $p$-adic algebraic formal stack $Y_d^{\leq h}$ of topologically finite type over $\mathcal{O}$ whose groupoid of $A$-valued points, for any $p$-adically complete $\mathcal{O}$-algebra topologically of finite type, is canonically equivalent to $Y^{\leq h}_d(A)$.
		
	For each Hodge type $\mu$ there exists a unique $\mathcal{O}$-flat closed substack $Y^\mu_d$ of $Y^{\leq h}_d$ with the property that the full subcategory $Y^\mu_d(A)$ of $Y^{\leq h}_d(A)$ consists of all crystalline Breuil--Kisin modules with height $\leq h$ and Hodge type $\mu$.
\end{proposition}
\begin{proof}
	The first part follows from \cite[\S4.5]{EG19}. There, algebraic stacks $\mathcal{C}^a_{\pi^\flat,s,d,h}$ over $\operatorname{Spec}\mathcal{O}/\pi^a$ are constructed \cite[4.5.8]{EG19} with $\pi^\flat = (\pi,\pi^{1/p},\ldots)$ and $s$ some sufficiently large integer. In the proof of \cite[4.5.15]{EG19} it is explained how $Y^{\leq h}_\mu \times_{\mathcal{O}} \mathcal{O}/\pi^a$ can be realised as a closed substack of $\mathcal{C}^a_{\pi^\flat,s,d,h}$. The second part is \cite[4.8.2]{EG19}.
\end{proof}

We conclude with a useful lemma giving a description of the points of $Y^\mu_d$ valued in a finite local $\mathbb{F}$-algebra:
\begin{lemma}\label{lift}
	Assume that $\mu_{\kappa} \subset [0,h]$ for each $\kappa$ and suppose that $(\overline{\mathfrak{M}},\overline{\sigma})$ corresponds to an $\overline{A}$-valued point of $\overline{Y}^\mu_d$ for $\overline{A}$ some finite local $\mathbb{F}$-algebra. Then there exists a local finite flat $\mathcal{O}$-algebra $A$ with $\overline{A} = A \otimes_{\mathcal{O}} \mathbb{F}$ and an $A$-valued point $(\mathfrak{M},\sigma)$ of $Y^\mu_d$ with special fibre $(\overline{\mathfrak{M}},\overline{\sigma})$.
\end{lemma}
\begin{proof}
	Let $\mathbb{F}'/\mathbb{F}$ be a finite extension and write $R_{\overline{\rho}}$ for the framed $\mathcal{O}$-deformation ring corresponding to some $\overline{\rho}:G_K\rightarrow \operatorname{GL}_d(\mathbb{F}')$. In \cite[2.2.11]{B19} a projective $R_{\overline{\rho}}$-scheme $\mathcal{L}^{\leq h}_{\overline{\rho}}$ is constructed with $A'$-points, for $A'$ any $p$-adically complete $\mathcal{O}$-algebra, classifying pairs $(\mathfrak{M},\rho)$ with $\rho$ a framed deformation of $\overline{\rho}$ to $A'$ and $\mathfrak{M} \in Z^{\leq h}_d(A')$ satisfying
	$$
	\mathfrak{M} \otimes_{\mathfrak{S}_A} W(C^\flat)_A = \rho \otimes_A W(C^\flat)_A
	$$
	so that $\varphi$ (induced semilinearly from that on $\mathfrak{M}$) is the identity on $\overline{\rho}$ and so that the $G_{K}$-action (induced semilinearly from that on $\rho$) satisfies
	$$
	(\sigma-1)(m) \in \mathfrak{M} \otimes_{\mathfrak{S}_A} [\pi^\flat]\varphi^{-1}(\mu)A_{\operatorname{inf},A},\qquad (\sigma_{\infty}-1)(m) =0
	$$
	for every $\sigma \in G_K,\sigma_{\infty} \in G_{K_\infty}$ and $m \in \mathfrak{M}$.
	The morphism $\mathcal{L}^{\leq h}_{\overline{\rho}} \rightarrow Y^{\leq h}_d$ given by $(\mathfrak{M},\rho) \mapsto (\mathfrak{M},\sigma)$ with $\sigma$ the $G_K$-action induced by $\rho$ is easily seen to be formally smooth. Therefore, the preimage $\mathcal{L}^{\mu}_{\overline{\rho}}$ of $Y^\mu_d$ in $\mathcal{L}^{\leq h}_{\overline{\rho}}$ is $\mathcal{O}$-flat. The map $\mathcal{L}^{\leq h}_{\overline{\rho}} \rightarrow \operatorname{Spec}R_{\overline{\rho}}$ becomes a closed immersion after inverting $p$ \cite[2.2.14]{B19}. Theorem~\ref{crys} therefore implies that $\mathcal{L}^{\mu}_{\overline{\rho}}[\frac{1}{p}] = \operatorname{Spec}R^\mu_{\overline{\rho}}$ where $R^\mu_{\overline{\rho}}$ is the reduced $\mathcal{O}$-flat quotient of $R_{\overline{\rho}}$ classifying crystalline representations of Hodge type $\mu$ \cite[3.3.8]{Kis08}. In particular, $\mathcal{L}^{\mu}_{\overline{\rho}}$ is reduced. 
	
	Now apply the above construction with $\overline{\rho} = (\overline{\mathfrak{M}} \otimes_{\mathfrak{S}_{\overline{A}}} W(C^\flat)_{\overline{A}})^{\varphi=1} \otimes_{\overline{A}} \mathbb{F}'$ where $\mathbb{F}'$ denotes the residue field of $\overline{A}$. Then $(\overline{\mathfrak{M}},\overline{\sigma})$ induces an $\overline{A}$-valued point of $\mathcal{L}^{\mu}_{\overline{\rho}}$. Applying \cite[4.1.2]{B19} to the local ring of $\mathcal{L}^{\mu}_{\overline{\rho}}$ at this points produces a finite flat $\mathcal{O}$-algebra $A$ with $A\rightarrow \overline{A}$ and an $A$-valued point of $\mathcal{L}^{\leq \mu}_{\overline{\rho}}$ pulling back to our $\overline{A}$-valued point. The image of this $A$-valued point in $Y^\mu_d$ then corresponds to $(\mathfrak{M},\sigma)$ as desired.
\end{proof}

\begin{corollary}\label{dim}
	Assume $\mu_{\kappa} \subset [0,h]$ for each $\kappa$. Then $Y^{\mu}_d$ has relative dimension $\sum_{\kappa} \operatorname{dim} G/P_{\mu_{\kappa}}$ over $\mathcal{O}$.
\end{corollary}
\begin{proof}
	We saw in the proof of Lemma~\ref{lift} that $\mathcal{L}^{\mu}_{\overline{\rho}} \rightarrow Y^\mu_d$ is formally smooth with fibres classifying framings of the corresponding Galois representation, and so of relative dimension $d^2$. Hence $Y^\mu_d$ has dimension (in the sense of e.g. \cite[0DRE]{stacks-project}) $\operatorname{dim}\mathcal{L}^\mu_{\overline{\rho}} - d^2$ at the image of the closed point of $\mathcal{L}^{\mu}_{\overline{\rho}}$. Since $\mathcal{L}^\mu_{\overline{\rho}}[\frac{1}{p}] = \operatorname{Spec}R^{\mu}_{\overline{\rho}}$ it follows from \cite[3.3.8]{Kis08} that $\mathcal{L}^{\mu}_{\overline{\rho}}$ has relative dimension $d^2 + \sum_{\kappa} \operatorname{dim} G/P_{\mu_{\kappa}}$ over $\mathcal{O}$. 
\end{proof}
\section{Naive Galois actions}\label{naivesection}

In this section we consider the morphism $Y^{\leq h}_d \rightarrow Z^{\leq h}_d$ which forgets the Galois action. More precisely, we consider its base-change $Y^{\leq h}_d \times_{Z^{\leq h}_d} \widetilde{Z}^{\leq h,N}_d \rightarrow \widetilde{Z}^{\leq h,N}_d$ for $N>>0$, and show this is an isomorphism over certain closed subschemes in the special fibre of $\widetilde{Z}^{\leq h,N}_d$.
 
\begin{construction}
	The aim is to establish conditions which allow the following ``naive'' crystalline $G_K$-action on $(\mathfrak{M},\beta) \in \widetilde{Z}^{\leq h,N}_d(A)$ to be perturbed into one which is $\varphi$-equivariant. Let $\sigma_{\operatorname{naive},\beta}$ denote the continuous $A_{\operatorname{inf},A}$-semilinear action of $G_K$ on $\mathfrak{M} \otimes_{\mathfrak{S}_A} A_{\operatorname{inf},A}$ obtained from the coordinate-wise action on $A_{\operatorname{inf},A}^d$ via the identification
	$$
	\mathfrak{M}^\varphi \otimes_{\mathfrak{S}_A} A_{\operatorname{inf},A} \cong A_{\operatorname{inf},A}^d
	$$
	induced by $\varphi(\beta)$.  Thus, $\sigma_{\operatorname{naive},\beta}$ is uniquely determined as the semilinear $G_K$-action fixing $\varphi(\beta)$.
\end{construction}

\begin{sub}\label{Grnablasigmar}
Let us fix integers $0 \leq r_\kappa \leq h$ for each $\kappa$. Then we consider the closed subfunctor $\operatorname{Gr}_{\mathcal{O}}^{\nabla_{\sigma},r} \subset \operatorname{Gr}_{\mathcal{O}}^{\nabla_{\sigma}}$ defined by requiring that
$$
\mathfrak{S}_A^d \subset \mathcal{E} \subset \prod_\kappa E_{\kappa}(u)^{-r_\kappa}\mathfrak{S}_A^d
$$
Define $\widetilde{Z}^{\nabla_{\sigma},r}_{d}$ as the preimage of $\operatorname{Gr}^{\nabla_{\sigma},r}_{\mathcal{O}}$ under the morphism $\Phi: \widetilde{Z}^{\leq h,N}_{d} \rightarrow \operatorname{Gr}_\mathcal{O}$ (for the moment $N$ is arbitrary). Then the $A$-points of $\widetilde{Z}^{\nabla_{\sigma},r}_{d}$ for $A$ a $p$-adically complete $\mathcal{O}$-algebra of topologically finite type are precisely those $(\mathfrak{M},\beta) \in \widetilde{Z}^{\leq h,N}_d(A)$ satisfying
	\begin{enumerate}
		\item For every $\sigma \in G_K$
		$$(\sigma_{\operatorname{naive},\beta}-1)(\mathfrak{M}) \subset \mathfrak{M} \otimes_{\mathfrak{S}_A} [\pi^\flat]\varphi^{-1}(\mu)A_{\operatorname{inf},A,i}$$ 
		See Remark~\ref{explanation}.
		\item For each $i=1,\ldots,f$ 
		$$
		\prod_j E_{\kappa}(u)^{r_\kappa} \mathfrak{M}_i \subset \mathfrak{M}^\varphi_i \subset \mathfrak{M}_i
		$$
	\end{enumerate}
\end{sub}
\begin{proposition}\label{isom}
	Assume that $\sum_{\kappa|_k = \kappa_0} r_{\kappa} \leq e+p-1$ for each $\kappa_0$ with at least one inequality strict. Then
	$$
	Y^{\leq h}_d \times_{Z^{\leq h}_d} \left( \widetilde{Z}^{\nabla_\sigma,r}_d \otimes_{\mathcal{O}} \mathbb{F} \right) \rightarrow \widetilde{Z}^{\nabla_\sigma,r}_d \otimes_{\mathcal{O}} \mathbb{F}
	$$
	is an isomorphism.
\end{proposition}
\begin{proof}
	By Lemma~\ref{mono} is enough to show that this morphism induces equivalences on groupoids of $\overline{A}$-valued points for any local finite type $\mathbb{F}$-algebra $\overline{A}$. Equivalently, we must show that for any $(\mathfrak{M},\beta) \in \widetilde{Z}^{\nabla_\sigma,r}_d(\overline{A})$ there exists a unique action $\sigma$ of $G_K$ making $(\mathfrak{M},\sigma)$ into an object of $Y^{\leq h}_d$. Existence implies essential surjectivity on $\overline{A}$-valued points and full-faithfullness follows from the uniqueness.
	
	Write $\operatorname{Hom}(\mathfrak{M},\mathfrak{M})$ for the $\mathfrak{S}_A$-module of $\mathfrak{S}_A$-linear endomorphisms of $\mathfrak{M}$ and equip $\operatorname{Hom}(\mathfrak{M},\mathfrak{M})$ with the Frobenius $\varphi_{\operatorname{Hom}}$ given by $h \mapsto \varphi \circ h \circ \varphi^{-1}$. If we identify $\operatorname{Hom}(\mathfrak{M},\mathfrak{M})$ with matrices in $\mathfrak{S}_A$ using the basis $\beta$ then $\varphi_{\operatorname{Hom}}$ acts by
	$$
	M\mapsto C\varphi(M)C^{-1}
	$$
	where $\varphi$ acts entrywise on the matrix $M$ and $C$ is such that $\varphi(\beta)=\beta C$. The following claim shows that, after extending scalars to $A_{\operatorname{inf},A}$, this operator is topologically nilpotent on matrices with entries divisible by $[\pi^\flat]\varphi^{-1}(\mu)$. 
	
	\begin{claim}
		Set $\mathcal{H} :=\operatorname{Hom}(\mathfrak{M},\mathfrak{M}) \otimes [\pi^\flat]\varphi^{-1}(\mu)A_{\operatorname{inf},\overline{A}}$. Then $\mathcal{H}$ is $\varphi_{\operatorname{Hom}}$-stable and $\varphi_{\operatorname{Hom}}$ is topologically nilpotent on $\mathcal{H}$.
	\end{claim}
	\begin{proof}[Proof of claim]
		Recall that $\mathfrak{M}_{\kappa_0} \subset \mathfrak{M}$ is the submodule on which $W(k)$ acts via $\kappa_0$ and $\varphi$ on $\mathfrak{M}$ restricts to a semilinear map $\mathfrak{M}_{\kappa_0} \rightarrow \mathfrak{M}_{\kappa_0 \circ \varphi^{-1}}[\frac{1}{E(u)}]$. Our assumption on $\mathfrak{M}$ implies $(\prod_{\kappa|_k = \kappa_0} E_\kappa(u)^{r_\kappa} )\mathfrak{M}_{\kappa_0} \subset \mathfrak{M}^\varphi_{\kappa_0}$ and therefore
		$$ 
		\varphi_{\operatorname{Hom}}(\operatorname{Hom}(\mathfrak{M},\mathfrak{M})_{\kappa_0}) \subset \left( \prod_{\kappa|_k = \kappa_0} E_\kappa(u)^{-r_\kappa}\right)\operatorname{Hom}(\mathfrak{M},\mathfrak{M})_{\kappa_0 \circ \varphi^{-1}}
		$$
		Since $\frac{[\pi^\flat]^p\mu }{[\pi^\flat] \varphi^{-1}(\mu)} A_{\operatorname{inf},\overline{A}} = [\pi^\flat]^{p-1}E(u)A_{\operatorname{inf},\overline{A}}$ we also have $$
		\varphi_{\operatorname{Hom}}(\mathcal{H}_{\kappa_0}) \subset [\pi^\flat]E(u)\left( \prod_{\kappa|_k = \kappa_0} E_\kappa(u)^{-r_\kappa}\right) \mathcal{H}_{\kappa_0 \circ \varphi^{-1}} = u^{e+p-1-\sum_j r_{ij}}\mathcal{H}_{i-1}
		$$
		(the last equality uses that $\overline{A}$ is an $\mathbb{F}$-algebra) and so, as $\sum_{\kappa_{\kappa|_k =\kappa_0}} r_{\kappa} \leq e+p-1$, it follows that $\mathcal{H}$ is $\varphi_{\operatorname{Hom}}$-stable. Since the inequality is strict at least once we have $\varphi_{\operatorname{Hom}}(\mathcal{H}_\kappa) \subset u\mathcal{H}_{\kappa_0 \circ \varphi^{-1}}$ for at least one $\kappa_0$. In particular $\varphi_{\operatorname{Hom}}$ is topologically nilpotent.
	\end{proof}
	Set $\mathfrak{M}_{\operatorname{inf}} = \mathfrak{M} \otimes_{\mathfrak{S}_A} A_{\operatorname{inf},A}$. Note that for each $\sigma \in G_K$ the endomorphism $\sigma_{\operatorname{naive},\beta}$ is $\sigma$-semilinear and so defines an element of 
	$$
\operatorname{Hom}(\mathfrak{M}_{\operatorname{inf}},\mathfrak{M}_{\operatorname{inf}}^\sigma)
$$
where $\mathfrak{M}_{\operatorname{inf}}^\sigma := \mathfrak{M}_{\operatorname{inf}} \otimes_{A_{\operatorname{inf},A},\sigma} A_{\operatorname{inf},A}$. Any such semilinear map is determined by where $\beta$ is sent; thus we obtain an (additive) identification $\operatorname{Hom}(\mathfrak{M}_{\operatorname{inf}},\mathfrak{M}_{\operatorname{inf}}^\sigma) = \operatorname{Hom}(\mathfrak{M},\mathfrak{M}) \otimes_{\mathfrak{S}_A} A_{\operatorname{inf},A}$ which identifies $\operatorname{Hom}(\mathfrak{M}_{\operatorname{inf}},[\pi^\flat]\varphi^{-1}(\mu)\mathfrak{M}_{\operatorname{inf}}^\sigma)$ and $\mathcal{H}$. Via this identification we view $\varphi_{\operatorname{Hom}}$ as a Frobenius on $\operatorname{Hom}(\mathfrak{M}_{\operatorname{inf}},\mathfrak{M}_{\operatorname{inf}}^\sigma)$. We claim that
$$
\varphi_{\operatorname{Hom}}^n(\sigma_{\operatorname{naive},\beta}) \in \operatorname{Hom}(\mathfrak{M}_{\operatorname{inf}},\mathfrak{M}_{\operatorname{inf}}^\sigma)
$$
for all $n\geq 0$ and that this sequence converges to a $\sigma$-semilinear endomorphism which we simply write as $\sigma$. By construction this endomorphism is $\varphi$-equivariant. To see this claim note that by assumption $\sigma_{\operatorname{naive},\beta} - 1^\sigma \in \operatorname{Hom}(\mathfrak{M}_{\operatorname{inf}},[\pi^\flat]\varphi^{-1}(\mu)\mathfrak{M}_{\operatorname{inf}}^\sigma)$ where $1^\sigma$ the $\sigma$-semilinear extension of the map $\beta \mapsto \beta$. Since we can write
$$
\operatorname{lim}_{n\rightarrow \infty} \varphi_{\operatorname{Hom}}^n(\sigma_{\operatorname{naive},\beta}) = \sigma_{\operatorname{naive},\beta} + \sum_{n \geq 1} \left(  \varphi^n_{\operatorname{Hom}}(\sigma_{\operatorname{naive},\beta}-1^\sigma) - \varphi^{n-1}_{\operatorname{Hom}}(\sigma_{\operatorname{naive},\beta}-1^{\sigma}) \right)
$$
the claimed convergence follows from the topological nilpotence of the operator $\varphi_{\operatorname{Hom}}$ on $\operatorname{Hom}(\mathfrak{M}_{\operatorname{inf}},[\pi^\flat]\varphi^{-1}(\mu)\mathfrak{M}_{\operatorname{inf}}^\sigma) = \mathcal{H}$.
This formula also shows that $\sigma(x) -x \in \mathfrak{M} \otimes_{\mathfrak{S}_A} [\pi^\flat]\varphi^{-1}(\mu)A_{\operatorname{inf},A}$ for each $x \in \mathfrak{M}$. Since $\sigma_{\operatorname{naive},\beta}$ defines a $G_K$-action so to does $\sigma$. Similarly, continuity of $\sigma_{\operatorname{naive},\beta}$ implies continuity of $\sigma$.

Finally, to see uniqueness suppose that $\sigma'$ was another such $G_K$-action. Then for each $\sigma \in G_K$ one has
$$
\sigma-\sigma' \in \operatorname{Hom}(\mathfrak{M}_{\operatorname{inf}},[\pi^\flat]\varphi^{-1}(\mu)\mathfrak{M}_{\operatorname{inf}}^\sigma) = \mathcal{H}
$$
Since $\sigma - \sigma'$ is $\varphi_{\operatorname{Hom}}$-fixed the topological nilpotence of $\varphi_{\operatorname{Hom}}$ on $\mathcal{H}$ implies $\sigma-\sigma' = 0$.
\end{proof}

\section{Comparison with local models}\label{comparison}
Set $\overline{Y}^{\mu}_d = Y^\mu_d \otimes_{\mathcal{O}} \mathbb{F}$.
\begin{theorem}\label{factor}
	Assume that $\mu_{\kappa} \subset [0,r_\kappa]$ for integers $r_{\kappa} \leq h$ satisfying 
	$$
	\sum_{\kappa|_k =\kappa_0} r_{\kappa} \leq \frac{p-1}{\nu_{\kappa_0}} +1, \qquad \nu_{\kappa_0} = \operatorname{max}_{\kappa|_k =\kappa_0}\lbrace v_\pi(\pi_{\kappa} - \pi_{\kappa'}) \rbrace
	$$
	for all $\kappa_0:k\rightarrow \mathbb{F}$. Then the composite
	$$
	\overline{Y}^\mu_d \times_{Z^{\leq h}_d} \widetilde{Z}^{\leq h,N}_d \rightarrow \widetilde{Z}^{\leq h,N}_d \xrightarrow{\Psi} \operatorname{Gr}_{\mathcal{O}}
	$$
	factors through $\overline{M}_{-w_0\mu}$ (recall this notation from Lemma~\ref{dual2})
\end{theorem}

\begin{remark}\label{tame}
	If $K$ is tamely ramified, i.e. if $e$ is not divisible by $p$, then $\pi_{\kappa} - \pi_{\kappa'}$ generates the same ideal of $\mathcal{O}_K$ as $\pi$ whenever $\kappa' \neq \kappa$. Therefore $v_{\kappa_0} = 0$ in this case. To see this consider the $\pi$-adic valuation of $\frac{d}{du}\kappa_0(E(u))|_{u=\pi_{\kappa}} = \prod_{\kappa \neq \kappa',\kappa'|_k = \kappa_0} (\pi_{\kappa} - \pi_{\kappa'})$.
\end{remark}

The following proposition is the key technical result which goes into the proof of the theorem. It is a reworking of techniques originally developed in \cite{GLS,GLS15}.
\begin{proposition}\label{GLS}
	Let $A$ be a finite flat $\mathcal{O}$-algebra and suppose $(\mathfrak{M},\sigma,\beta) \in Y^{\mu}_d(A)$. Define 
	$$
	M_{\kappa} = \mathfrak{M}^\varphi /E_{\kappa}(u)
	$$
	Use the $\varphi(\mathfrak{S}_A)$-basis $\varphi(\beta)$ to define a section $s$ of $\varphi(\mathfrak{M}) \rightarrow \mathfrak{M}^\varphi \rightarrow M_{\kappa}$. Then there exists a filtration $\operatorname{Fil}^{\bullet}_{\kappa}$ on $M_{\kappa}$ by $A$-submodules with $p$-torsionfree graded pieces such that
	$$
	\sum_{n=0}^{r_\kappa} E_{\kappa}(u)^{r_\kappa-n} \mathfrak{S}_A\operatorname{Fil}^n_{\kappa} + \mathfrak{M}^\varphi_{\operatorname{err},\kappa} = \mathfrak{M}^\varphi \cap E_{\kappa}(u)^{r_\kappa}\mathfrak{M} + \mathfrak{M}_{\operatorname{err},\kappa}^\varphi 
	$$
	when $\operatorname{Fil}^n_{\kappa}$ is viewed as a submodule of $\mathfrak{M}^\varphi$ via $s$ and $\mathfrak{M}_{\operatorname{err},\kappa}^\varphi := \sum_{l=1}^{p-1} \pi^{p-l} E_{\kappa}(u)^{l}\mathfrak{M}^\varphi$. 
\end{proposition}

Note that, by construction, the image of the section $s$ generates $\mathfrak{M}^\varphi$ over $\mathfrak{S}_A$.
\begin{proof}
	First we define the filtration $\operatorname{Fil}_{\kappa}^\bullet$. Recall that we equipped $M :=\mathfrak{M}^\varphi/E(u)$ with the filtration whose $n$-th piece is the image of $\mathfrak{M}^\varphi \cap E(u)^n\mathfrak{M}$. Then $D_K:= M[\frac{1}{p}]$ is a filtered $A\otimes_{\mathbb{Z}_p} K$-module and can be written as $\prod_{\kappa} D_{K,\kappa}$ with each $D_{K,\kappa}$ a filtered $A[\frac{1}{p}]$-module. The composite $\mathfrak{M}^\varphi \rightarrow D_K \rightarrow D_{K,\kappa}$ is obtained by base-change along the map $\mathfrak{S} \rightarrow A[\frac{1}{p}]$ given by $u \mapsto \pi_{\kappa} $. Therefore, its kernel is $E_{\kappa}(u)\mathfrak{M}^\varphi$. This means $M_{\kappa}$ can be viewed as a submodule of $D_{K,\kappa}$ and $M_{\kappa}[\frac{1}{p}]  =D_{K,\kappa}$. Define 
	$$
	\operatorname{Fil}_{\kappa}^n = \operatorname{Fil}^n(D_{K,\kappa}) \cap M_{\kappa}
	$$
	The filtered pieces of $D_{K,\kappa}$ are $\mathbb{Q}_p$-vector spaces so the graded pieces of $\operatorname{Fil}^\bullet_{\kappa}$ are $p$-torsionfree. This also shows that $\operatorname{Fil}_{\kappa}^n[\frac{1}{p}] = \operatorname{Fil}^n(D_{K,\kappa})$ which proves Corollary~\ref{filtype} below.
	
	Next we use:
	\begin{claim}
		For $\overline{x} \in \operatorname{Fil}_{\kappa}^n$ with $n\leq p$ there exists $x_1,\ldots,x_{p-1} \in \mathfrak{M}^\varphi$ such that
		$$
		s(\overline{x}) + E_{\kappa}(u)\pi^{p-1}x_1 + \ldots + E_{\kappa}(u)^{p-1} \pi x_{p-1} \in \mathfrak{M}^\varphi \cap E_{\kappa}(u)^{n}\mathfrak{M}
		$$
	\end{claim}
	\begin{proof}[Proof of Claim]
	This follows from results in \cite[\S5]{B19b}. To apply these first note that in \emph{loc. cit.} the embeddings $K\rightarrow E$ are indexed by integers $1\leq i \leq f$ and $1 \leq j \leq e$ so that $\kappa_{ij}|_k$ depends only on $i$. This labelling can be chosen so that $\kappa$ from the proposition equals $\kappa_{i1}$ for some $i$. In \cite[5.2.5]{B19b} it is shown that for any $x \in \varphi(\mathfrak{M})$ there exist $x_1,\ldots,x_{p-1} \in \mathfrak{M}^\varphi$ so that
	$$
	x^{(n)} - x + E_{1}(u)\pi^{p-1}x_1 + \ldots + E_{1}(u)^{p-1} \pi x_{p-1} \in E_1(u)^p \mathfrak{M}^\varphi \otimes_{\mathfrak{S}} S[\tfrac{1}{p}]
	$$
	where
	\begin{itemize}
		\item  $S$ is the ring defined in \cite[\S5.1]{B19b},
		\item $x^{(i)}$ is defined as in \cite[\S 5.2]{B19b},
		\item $E_1(u) = \prod_{i=1}^f E_{i1}(u) \in \mathfrak{S}_{\mathcal{O}}$ for $E_{ij}(u) := E_{\kappa_{ij}}(u)$.
	\end{itemize}
We apply this to $x = s(\overline{x})$. Then the image of $x$ in $D_{K,\kappa}$ is contained in $\operatorname{Fil}^n(D_{K,\kappa})$ and so \cite[5.2.2]{B19b} implies $x^{(n)}$ is contained in a submodule of $\mathfrak{M}^\varphi \otimes_{\mathfrak{S}} S[\frac{1}{p}]$ denoted $\operatorname{Fil}^{\lbrace n,0,\ldots,0\rbrace}$. In \cite[5.1.3]{B19b} it is shown that $\operatorname{Fil}^{\lbrace n,0,\ldots,0\rbrace} \cap \mathfrak{M}^\varphi = \mathfrak{M}^\varphi \cap E_{1}(u)^n\mathfrak{M}$. Therefore
	$$
	s(\overline{x}) + E_{1}(u)\pi^{p-1}x_1 + \ldots + E_{1}(u)^{p-1} \pi x_{p-1} \in \mathfrak{M}^\varphi\cap E_{1}(u)^{n}\mathfrak{M} \subset \mathfrak{M}^\varphi \cap E_{\kappa}(u)^n\mathfrak{M}
	$$
	(the inclusion following because $E_\kappa(u)$ divides $E_1(u)$). Under the identification $\mathfrak{M}^\varphi = \prod_{\kappa_0} \mathfrak{M}^\varphi_{\kappa_0}$ the $\kappa_0$-th part of $\mathfrak{M}^\varphi \cap E_{\kappa}(u)^n\mathfrak{M}$ is just $\mathfrak{M}^\varphi_{\kappa_0}$ for $\kappa_0 \neq \kappa|_k$. Therefore, in the above identity we can replace each $E_1(u)$ with $E_\kappa(u)$, and the claim follows.
	\end{proof}
	The claim shows that
	$$
\underbrace{\sum_{n=0}^m E_{\kappa}(u)^{m-n}\mathfrak{S}_A \operatorname{Fil}^n_{\kappa}}_{:=Y_m} + \mathfrak{M}^\varphi_{\operatorname{err},\kappa}  \subset \mathfrak{M}^\varphi \cap E_{\kappa}(u)^m \mathfrak{M}+ \mathfrak{M}^\varphi_{\operatorname{err},\kappa}
	$$
	for any $0 \leq m\leq r_\kappa$ and we want to prove the opposite inclusion for $0 \leq m \leq r_\kappa$ by induction on $m$. When $m=0$ this is clear since both sides equal $\mathfrak{M}^\varphi$ (recall that the section $s$ was chosen so that $s(M_{\kappa})$ generates $\mathfrak{M}^\varphi$ over $\mathfrak{S}_A$). For $m>0$ note that the image of $\mathfrak{M}^\varphi \cap E_{\kappa}(u)^m\mathfrak{M}$ in $M_{\kappa}$ is contained in $\operatorname{Fil}_\kappa^m$, while $\operatorname{Fil}^m_\kappa$ equals the image of $Y_m$. The above inclusion therefore shows these images are equal. As a consequence, if $x \in \mathfrak{M}^\varphi \cap E_{\kappa}(u)^m \mathfrak{M}$ then there exists $x' \in Y_m$ so that
	$$
	\begin{aligned}
	x- x' &\in E_{\kappa}(u)\mathfrak{M}^\varphi \cap \left( \mathfrak{M}^\varphi \cap E_{\kappa}(u)^m\mathfrak{M} + \mathfrak{M}^{\varphi}_{\operatorname{err},\kappa} \right)\\
	 &= \left( E_{\kappa}(u) \mathfrak{M}^\varphi \cap E_{\kappa}(u)^m\mathfrak{M}\right)+ \mathfrak{M}^{\varphi}_{\operatorname{err},\kappa}\\
	 &= E_{\kappa}(u) \left( \mathfrak{M}^\varphi \cap E_\kappa(u^{m-1})\mathfrak{M}\right) + \mathfrak{M}^{\varphi}_{\operatorname{err},\kappa}
\end{aligned}
	$$
	The second equality uses that $\mathfrak{M}^\varphi_{\operatorname{err},\kappa} \subset E_\kappa(u)\mathfrak{M}^\varphi$. The inductive hypothesis therefore gives that $x-x' \in E_{\kappa}(u)Y_{m-1} + \mathfrak{M}^\varphi_{\operatorname{err},\kappa}$. Since $E_{\kappa}(u)Y_{m-1} \subset Y_m$ it follows that $x \in Y_m+\mathfrak{M}^{\varphi}_{\operatorname{err},\kappa}$ as desired.
\end{proof}
\begin{corollary}\label{filtype}
	The graded pieces of $\operatorname{Fil}^\bullet_{\kappa}$ become $A[\frac{1}{p}]$-projective after inverting $p$ and $\operatorname{Fil}_{\kappa}^\bullet[\frac{1}{p}]$ has type $-w_0\mu_{\kappa} = (-\mu_{\kappa,d}\geq \ldots \geq -\mu_{\kappa,1})$.
\end{corollary}
\begin{proof}[Proof of Theorem~\ref{factor}]
	By Corollary~\ref{factorising} it suffices to prove the factorisation on the level of $\overline{A}$-valued points for $\overline{A}$ any finite local $\mathbb{F}$-algebra. Let $(\overline{\mathfrak{M}},\overline{\sigma},\overline{\beta})$ be such a point. Applying Lemma~\ref{lift} we obtain a local finite flat $\mathcal{O}$-algebra $A$ with a map $A\rightarrow \overline{A}$ and $(\mathfrak{M},\sigma) \in Y^{\mu}_d(A)$ lifting $(\overline{\mathfrak{M}},\overline{\sigma})$. Additionally, choose an $\mathfrak{S}_A$-basis $\beta$ lifting $\overline{\beta}$. We will then be done if we can show that the special fibre of $(\mathfrak{M},\sigma,\beta)$ is mapped into $\overline{M}_{-w_0\mu}$ by $\Psi$. We can assume that $\overline{A} = A \otimes_{\mathcal{O}} \mathbb{F}$.
	
	Applying Proposition~\ref{GLS} for each $\kappa$ we obtain filtrations $\operatorname{Fil}^\bullet_{\kappa}$. Define $\mathcal{E}$ by  
	$$
	\left( \prod_\kappa E_\kappa(u)^{r_\kappa} \right) \mathcal{E} = \bigcap_{\kappa} \left(  \sum_{n=0}^{r_\kappa} E_{\kappa}(u)^{r_\kappa-n}\mathfrak{S}_A\operatorname{Fil}^n_{\kappa} \right)
	$$
	As in Proposition~\ref{GLS} the $\operatorname{Fil}_\kappa^n$'s are viewed as submodules of $\mathfrak{M}^\varphi$ using the basis $\varphi(\beta)$. Corollary~\ref{filtype} together with part (2) of Lemma~\ref{points} (taking $n_\kappa=r_\kappa$) shows that $\mathcal{E}[\frac{1}{p}]$ defines an $A[\frac{1}{p}]$ point of $M_{-w_0\mu}$ under the identification $\mathfrak{M}^\varphi= \mathfrak{S}_A^d$ induced by $\varphi(\beta)$.
	Note, however, that it is not a priori clear $\mathcal{E}$ defines an $A$-valued point. We will be done if we can show this is the case, and if we can show that $\mathcal{E} \otimes_{\mathcal{O}} \mathbb{F} = \mathfrak{M} \otimes_{\mathcal{O}} \mathbb{F}$.
	
	 We begin with the second assertion. Take $z \in \left( \prod_\kappa E_\kappa(u)^{r_\kappa} \right)\mathfrak{M}$. Then $z \in \mathfrak{M}^\varphi \cap E_\kappa(u)^{r_\kappa} \mathfrak{M}$ for each $\kappa$ and so Proposition~\ref{GLS} ensures the existence of $m_{\kappa} \in \mathfrak{M}^\varphi_{\operatorname{err},\kappa}$ such that 
	$$
	z-m_{\kappa} \in \left(  \sum_{n=0}^{r_\kappa} E_{\kappa}(u)^{r_\kappa-n}\mathfrak{S}_A\operatorname{Fil}^n_{\kappa} \right)
	$$
	We claim there then exists $m \in \pi\mathfrak{M}^\varphi$ such that 
	$$
	m \equiv m_{\kappa}~\operatorname{mod}~E_{\kappa}(u)^{r_{\kappa}}\mathfrak{M}^\varphi
	$$
	for each $\kappa$. Since $\mathfrak{M}^\varphi$ is $\mathfrak{S}_A$-free this claim follows from Lemma~\ref{claim} below. This is where we use the bound on the $r_{\kappa}$. Since 
	$$
	E_{\kappa}(u)^{r_{\kappa}}\mathfrak{M}^\varphi \subset \left(  \sum_{n=0}^{r_\kappa} E_{\kappa}(u)^{r_\kappa-n}\mathfrak{S}_A\operatorname{Fil}^n_{\kappa} \right)
	$$ 
	for each $\kappa$ (due to the filtration $\operatorname{Fil}_{\kappa}^n$ being concentrated in degrees $[0,r_{\kappa}]$ we have $\operatorname{Fil}^0_\kappa = M_\kappa$) it follows that $z-m \in \left( \prod_\kappa E_\kappa(u)^{r_\kappa} \right)\mathcal{E}$. Since $m \in \pi\mathfrak{M}^\varphi$ the image of $z$ in $\mathfrak{M}^\varphi \otimes_{\mathcal{O}}\mathbb{F}$ is contained in the image of $\left( \prod_\kappa E_\kappa(u)^{r_\kappa} \right)\mathcal{E}$. A symmetrical argument shows also that if $z \in \left( \prod_\kappa E_\kappa(u)^{r_\kappa} \right)\mathcal{E}$ then its image in $\mathfrak{M}^\varphi \otimes_{\mathcal{O}} \mathbb{F}$ is contained in the image of $\left( \prod_\kappa E_\kappa(u)^{r_\kappa} \right)\mathfrak{M}$. 
	
	Next we show that $\mathcal{E}$ is $\mathfrak{S}_A$-projective. This is equivalent to $\mathfrak{M}^\varphi / \left( \prod_\kappa E_\kappa(u)^{r_\kappa} \right)\mathcal{E}$ being $A$-projective. From the definitions we see that $\mathfrak{M}^\varphi /\left( \prod_\kappa E_\kappa(u)^{r_\kappa} \right)\mathcal{E}$ is $p$-torsionfree. This means that $\left( \prod_\kappa E_\kappa(u)^{r_\kappa} \right)\mathcal{E} \otimes_{\mathcal{O}} \mathbb{F}$ equals its image in $\mathfrak{M}^\varphi \otimes_{\mathcal{O}} \mathbb{F}$. It also means that, by \cite[00ML]{stacks-project}, $A$-projectivity of $\mathfrak{M}^\varphi / \left( \prod_\kappa E_\kappa(u)^{r_\kappa} \right)\mathcal{E}$ follows from $A \otimes_{\mathcal{O}} \mathbb{F}$-projectivity of $\mathfrak{M}^\varphi \otimes_{\mathcal{O}} \mathbb{F} / \left( \prod_\kappa E_\kappa(u)^{r_\kappa} \right)\mathcal{E} \otimes_{\mathcal{O}}\mathbb{F}$. But we saw in the previous paragraph that $\left( \prod_\kappa E_\kappa(u)^{r_\kappa} \right)\mathcal{E} \otimes_{\mathcal{O}} \mathbb{F} = \left( \prod_\kappa E_\kappa(u)^{r_\kappa} \right)\mathfrak{M} \otimes_{\mathcal{O}} \mathbb{F}$. Since $\mathfrak{M}^\varphi /\left( \prod_\kappa E_\kappa(u)^{r_\kappa} \right)\mathfrak{M}$ is $A$-projective the claimed $\mathfrak{S}_A$-projectivity of $\mathcal{E}$ follows. This establishes the two required conditions mentioned in the second paragraph, and therefore finishes the proof.
\end{proof}
\begin{lemma}\label{claim}
	Let $A$ be a finite flat $\mathcal{O}$-algebra and suppose 
	$$
	m_\kappa = \sum_{l=1}^{p} E_{\kappa}(u)^{p-l}  \pi^{l}m_{\kappa,l} 
	$$
	are given with $m_{\kappa,l}\in A$. Then there exists $m \in \pi\mathfrak{S}_A$ with $m \equiv m_{\kappa}$ modulo $E_\kappa(u)^{r_\kappa}$ for each $\kappa$.
\end{lemma}
\begin{proof}
Firstly, by choosing an $\mathcal{O}$-basis of $A$ we can reduce to the case $A = \mathcal{O}$. Secondly, we can fix $\kappa$ and assume that $m_{\kappa'} = 0$ for all $\kappa' \neq \kappa$. Using the identification $\mathfrak{S}_{\mathcal{O}} = \prod_{\kappa_0} \mathcal{O}[[u]]$ we are left proving that if $m \in \mathcal{O}[u]$ can be written as
$$
m = \sum_{l=1}^p (u-\pi_\kappa)^{p-l} \pi^{l} m_l,\qquad m_l \in \mathcal{O}
$$
then there exists $M \in \pi \mathcal{O}[[u]]$ with \begin{itemize}
	\item $M$ divisible by $(u-\pi_{\kappa'})^{r_{\kappa'}}$ for every $\kappa' \neq \kappa$ with $\kappa'|_k = \kappa|_k$.
	\item $M \equiv m$ modulo $(u-\pi_{\kappa})^{r_{\kappa}}$.
\end{itemize}
We will construct $M$ explicitly. For $\kappa'$ with $\kappa'|_k = \kappa|_k$ and $\kappa' \neq \kappa$ set
$$
X_{\kappa'} := \sum_{n=0}^{r_{\kappa}-1} \binom{r_{\kappa'}-1+j}{r_{\kappa'}-1} \frac{(u-\pi_{\kappa})^n}{(\pi_{\kappa} - \pi_{\kappa'})^{n+r_{\kappa'}}}
$$
Using the formal identity $\frac{1}{(1-y)^r} = \sum_{n=0}^\infty \binom{r-1+i}{r-1} y^i$ we see that
$$
X_{\kappa'}(u-\pi_{\kappa'})^{r_{\kappa'}} \equiv 1 \text{ modulo } (u-\pi_{\kappa})^{r_{\kappa}}
$$ 
Define $N \in E[u]$ to be the polynomial of degree $< r_\kappa$ when viewed as a polynomial in $(u-\pi_{\kappa})$, obtained by truncating  $m \prod_{\kappa' \neq \kappa} X_{\kappa'}$. Then $N \equiv m \prod_{\kappa' \neq \kappa} X_{\kappa'}$ modulo $(u-\pi_\kappa)^{r_\kappa}$ and so
$$
M := N \prod_{\kappa' \neq \kappa} (u-\pi_{\kappa'})^{r_{\kappa'}} 
$$
satisfies the two bullet points above. To finish it suffices to show that $N$, and hence $M$ also, is contained in $\pi\mathcal{O}[[u]]$.

For this view $N$ as a polynomial in $(u-\pi_\kappa)$. By assumption the coefficient of $(u-\pi_{\kappa})^n$ in $m$ has valuation $\geq p-n$. On the other hand, the coefficient of $(u-\pi_{\kappa})$ in $X_{\kappa'}$ has valuation $\geq -(n+r_{\kappa'})\nu$ for $\nu := \nu_{\kappa|_k}$. Since $\nu \geq 1$ we have $p-n \geq p - n \nu$ and as such the coefficient of $(u-\pi_{\kappa})^n$ in $m \prod_{\kappa' \neq \kappa} X_j$ has valuation
$$
\geq p -(\sum_{\kappa'\neq \kappa} r_{\kappa'} + n)\nu
$$
We will be done if $p -(\sum_{\kappa'\neq \kappa} r_{\kappa'} + n)\nu \geq 1$ for all $n =0,\ldots,r_{\kappa}-1$, i.e. if $p- (\sum_{\kappa'|_k = \kappa|_k} r_{\kappa'} -1)\nu \geq 1$. This is equivalent to asking that $\sum_{\kappa'|_k= \kappa|_k} r_{\kappa'} \leq \frac{p-1}{\nu} +1$ so we are done.
\end{proof}
\section{Lower bounds}

In this section we recall from \cite{GK15,EG19} the lower bound on the cycles appearing in the Breuil--M\'ezard conjecture attained when $d=2$. We do this in the context of cycles in deformation rings. We also give a minor improvement using the potential diagaonalisability established in \cite{B19b}.
\begin{sub}\label{absirred}
	First we recall that isomorphism classes of  absolutely irreducible $\mathbb{F}$-representations of $\operatorname{GL}_d(k)$ are in bijection with those tuples $\lambda = (\lambda_{\kappa_0})$ indexed by embeddings $\kappa_0:k\rightarrow \mathbb{F}$ for which
	$$
	\lambda_{\kappa_0,1} -\lambda_{\kappa_0,d} \leq p-1
	$$ 
	This bijection sends $(\lambda_{\kappa_0})$ onto the $\operatorname{GL}_d(k)$-representation obtained by evaluating the algebraic representation of $G=\operatorname{GL}_d$
	$$
	\bigotimes_{\kappa_0} \left( L(\lambda_{\kappa_0}) \otimes_{k,\kappa_0} \mathbb{F} \right)
	$$
	on $k$-points \cite{Her09}. Here $L(\lambda_{\kappa_0}) \subset H^0(G/P_{\lambda_{\kappa}},\mathcal{L}(\lambda_{\kappa})) \otimes_{\mathcal{O}_K} k$ denotes the unique irreducible algebraic $G$-submodule.
\end{sub}
\begin{sub}
	We also recall from the introduction the $\mathbb{F}$-representation $V(\mu,\tau)$ of $\operatorname{GL}_d(k)$ attached to any pair $(\mu,\tau)$ with $\mu$ a Hodge type with each $\mu_\kappa - \rho$ dominant  and $\tau$ an inertial type. Taking $\tau =1$ we obtain $V(\mu,1)$ by evaluating the algebraic representation
	$$
	\bigotimes_{\kappa} \left( H^0(G/P_{\mu_\kappa -\rho},\mathcal{L}(\mu_\kappa-\rho)) \otimes_{\mathcal{O}_K,\kappa} \mathbb{F} \right)
	$$
	(here we write $\kappa$ also for its composite with the surjection $\mathcal{O} \rightarrow \mathbb{F}$) on $k$-points. Since the exact definition of $V(\mu,\tau)$ will not be needed for $\tau \neq 1$ we refer to \cite[8.2]{EG19} for the general construction.
\end{sub}

\begin{lemma}\label{compare}
	Suppose $d=2$ and $\mu$ is a Hodge type with $\mu_\kappa \subset [0,r_\kappa]$ for $r_\kappa \geq 0$ satisfying
	$$
	\sum_{\kappa|_{k} = \kappa_0} r_\kappa \leq e+p-1
	$$
	Then the multiplicity of $\lambda$ in $V(\mu,1)$, for $\lambda$ an absolutely irreducible $\mathbb{F}$-representation of $\operatorname{GL}_d(k)$ corresponding to $(\lambda_{\kappa_0})$ under \ref{absirred}, equals the product
	$$
	\prod_{\kappa_0} m(\lambda_{\kappa_0}, (\mu_{\kappa})_{\kappa|_k = \kappa_0})
	$$
	from Proposition~\ref{upperbound}.
\end{lemma}
\begin{proof}
	Recall from Lemma~\ref{polynomialsss} that each $m(\lambda_{\kappa_0}, (\mu_{\kappa})_{\kappa|_k = \kappa_0})$ can be interpreted as the multiplicity of $H^0(\lambda_{\kappa_0})$ in $ \bigotimes_{\kappa|_k = \kappa_0} H^0(\mu_\kappa -\rho)$ where $H^0(-)$ denotes the generic fibre of $H^0(G/P_{-},\mathcal{L}(-))$. This coincides with the corresponding multiplicities on the special fibre, i.e. the multiplicity of $H^0(G/P_{\lambda_{\kappa_0}},\mathcal{L}(\lambda_{\kappa_0})) \otimes_{\mathcal{O}_K} k$ in
	$$
	\bigotimes_{\kappa|_k = \kappa_0}\left( H^0(G/P_{\lambda_{\mu_\kappa -\rho}},\mathcal{L}(\mu_\kappa-\rho)) \otimes_{\mathcal{O}_K} k\right)
	$$
	The product of these multiplicities therefore equals the multiplicity of $$
	\bigotimes_{\kappa_0} H^0(G/P_{\lambda_{\kappa_0}},\mathcal{L}(\lambda_{\kappa_0})) \otimes_{\mathcal{O}_K,\kappa_0} \mathbb{F}$$
	inside
	$$
	\bigotimes_{\kappa}\left( H^0(G/P_{\lambda_{\mu_\kappa -\rho}},\mathcal{L}(\mu_\kappa-\rho)) \otimes_{\mathcal{O}_K,\kappa} \mathbb{F}\right)
	$$
	The lemma will therefore follow if each $ H^0(G/P_{\lambda_{\kappa_0}},\mathcal{L}(\lambda_{\kappa_0})) \otimes_{\mathcal{O}_K} k$ is simple, i.e. equals $L(\lambda_{\kappa_0})$. Since $d=2$ this can be seen from the explicit description given in, for example, \cite[II.2.16]{Janbook}.
\end{proof}
\begin{notation}
	If $\lambda$ denotes an isomorphism class of absolutely irreducible $\mathbb{F}$-representation of $\operatorname{GL}_d(k)$ then we write $\widetilde{\lambda}$ for the Hodge type obtained as in \ref{tilde} for $(\lambda_{\kappa_0})$ the tuple corresponding to $\lambda$ in \ref{absirred}.
\end{notation}

\begin{sub}
	In the next proposition we fix a continuous homomorphism $\overline{\rho}:G_K \rightarrow \operatorname{GL}_d(\mathbb{F})$ and, as in the proof of Lemma~\ref{lift}, we write $R_{\overline{\rho}}$ for the $\mathcal{O}$-framed deformation ring of $\overline{\rho}$. If $(\mu,\tau)$ is a pair consisting of a Hodge type $\mu$ and an inertial type $\tau$ then we also write $R_{\overline{\rho}}^{\mu,\tau}$ for the unique reduced $\mathcal{O}$-flat quotient of $R_{\overline{\rho}}$ whose points valued in a finite extension of $E$ are correspond to potentially crystalline representations of type $(\mu,\tau)$.

We also say that an absolutely irreducible representation $\mathbb{F}$-representation $\lambda$ of $\operatorname{GL}_2(k)$ is non-Steinberg if $\lambda$ corresponds to a tuple $(\lambda_{\kappa_0})$ with $\lambda_{\kappa_0,1} - \lambda_{\kappa_0,2} \neq p-1$ for at least one $\kappa_0$.
\end{sub}\begin{proposition}\label{global}
	Assume $p>2$, $d=2$, and that $\mu_\kappa - \rho$ is dominant for each $\kappa$. Then
	\begin{enumerate}
		\item There are cycles $\mathcal{C}_{\overline{\rho},\lambda}$ in $\operatorname{Spec}R_{\overline{\rho}}$, indexed by isomorphism classes of absolutely irreducible $\mathbb{F}$-representations $\lambda$ of $\operatorname{GL}_d(k)$, such that, for any pair $(\mu,\tau)$, one has an inequality
		$$
		[\operatorname{Spec}R^{\mu,\tau}_{\overline{\rho}} \otimes_{\mathcal{O}} \mathbb{F}] \geq \sum_\lambda m(\lambda,\mu,\tau) \mathcal{C}_{\overline{\rho},\lambda}
		$$
		for $m(\lambda,\mu,\tau)$ the multiplicity of $\lambda$ in $V(\mu,\tau)$ (where $V(\mu,\tau)$ is the $\operatorname{GL}_d(k)$-representation attached to $\mu,\tau$).
		\item If $\lambda$ is non-Steinberg then $\mathcal{C}_{\overline{\rho},\lambda} = [\operatorname{Spec}R_{\overline{\rho}}^{\widetilde{\lambda},1} \otimes_{\mathcal{O}} \mathbb{F}]$.
	\end{enumerate}
\end{proposition}
This is the single point where the assumption $p>2$ arises.
\begin{proof}
	In \cite[8.6.6]{EG19} it is shown that if $\mathcal{X}^{\mu,\tau}_2$ denotes the closed algebraic substack of the Emerton--Gee stack $\mathcal{X}_2$, whose $A$-points (for $A$ a finite flat $\mathcal{O}$-algebra) correspond to potentially crystalline $G_K$-representations of type $(\mu,\tau)$, then there are top dimensional cycles $\mathcal{C}_{\lambda} \subset \overline{\mathcal{X}}_2$ such that
	$$
	[\overline{\mathcal{X}}^{\mu,\tau}_2] \geq \sum_\lambda m(\lambda,\mu,\tau) \mathcal{C}_\lambda
	$$ 
	Furthermore, the cycles $\mathcal{C}_{\lambda}$ are described explicitly in \cite[8.6.2]{EG19} (using results from \cite{CEGS19}). In particular, if $\lambda$ is non-Steinberg then $\mathcal{C}_{\lambda}$ is the irreducible component of $\overline{\mathcal{X}}_2$ labelled by $\lambda$ as in \cite[5.5.11]{EG19}. 
	
	As explained in \cite[8.3]{EG19} the above inequality implies part (1) of the proposition by pulling back along the formally smooth morphism $\operatorname{Spf}R_{\overline{\rho}} \otimes_{\mathcal{O}} \mathbb{F} \rightarrow \overline{\mathcal{X}}_2$. Likewise, part (2) will follow if we can show $[\overline{\mathcal{X}}^{\widetilde{\lambda},1}_2] = \mathcal{C}_{\lambda}$ for $\lambda$ non-Steinberg. Furthermore, \cite[8.3]{EG19} explains that this equality is implied by the assertion that 
	$$
	e(R_{\overline{\rho}}^{\widetilde{\lambda},1} \otimes_{\mathcal{O}}\mathbb{F}) =1
	$$
	for every $\overline{\rho}$ contained in a dense open subset of $\mathcal{C}_{\lambda}$, where $e(R)$ denotes the Hilbert--Samuel multiplicity of a local ring $R$. We do this by considering the dense open consisting of $\overline{\rho}$ maximally non-split of niveau $1$ and weight $\lambda$, which is indicated in part (2) of \cite[5.5.11]{EG19}. By \cite[5.5.4]{EG19} we know $e(R_{\overline{\rho}}^{\widetilde{\lambda},1} \otimes_{\mathcal{O}}\mathbb{F}) \geq 1$ for such $\overline{\rho}$. From the definition given in \cite[5.5.1]{EG19} we also know that if $\lambda$ is non-Steinberg then any such $\overline{\rho}$ is not of the form
	$$
	\psi \otimes \begin{pmatrix}
		1 & * \\ 0 & \chi_{\operatorname{cyc}}^{-1}
	\end{pmatrix}
	$$ 
	for $\psi$ an unramified character and $\chi_{\operatorname{cyc}}$ the mod $p$ cyclotomic character.
	
	It remains to prove that $e(R_{\overline{\rho}}^{\widetilde{\lambda},1} \otimes_{\mathcal{O}}\mathbb{F}) \leq1$ for $\overline{\rho}$ maximally non-split of niveau $1$ and weight $\lambda$. For this we recall \cite[3.5.5]{GK15} which asserts that for any $\overline{\rho}$ there is a unique set of integers $e_{\lambda}(\overline{\rho}) \geq 0$ such that, if every potentially crystalline lift of $\overline{\rho}$ of type $\mu,\tau$ is potentially diagonalisable, then
	$$
	e(R^{\mu,\tau}_{\overline{\rho}} \otimes_{\mathcal{O}} \mathbb{F}) = \sum_{\lambda} m(\lambda,\mu,\tau) e_\lambda(\overline{\rho})
	$$ 
	The integers $e_\lambda(\overline{\rho})$ are the Hilbert--Samuel multiplicities of the cycles $\mathcal{C}_{\lambda,\overline{\rho}}$ and the above discussion implies these are $1$ when $\lambda$ is non-Steinberg. Therefore, part (2) will follow if we can show every crystalline lift of $\overline{\rho}$ with Hodge type $\widetilde{\lambda}$ is potentially diagaonalisable. This is the main result of \cite{B19b} (which applies since $\overline{\rho}$ is not an unramified twist of $\chi_{\operatorname{cyc}}^{-1}$ by $1$).
\end{proof}
\section{Main result}

We can now prove our main result.

\begin{theorem}
	Assume $p>2$ and $d=2$. Let $\mu$ be a Hodge type with each $\mu_{\kappa} - \rho$ dominant and
	$$
	\sum_{\kappa_{\kappa|_k =\kappa_0}} (\mu_{\kappa,1}- \mu_{\kappa,2} -1) \leq p
	$$
	for each $\kappa_0:k\rightarrow \mathbb{F}$. Then
	$$
	[\operatorname{Spec}R^{\mu,1}_{\overline{\rho}} \otimes_{\mathcal{O}} \mathbb{F}] = \sum_\lambda m(\lambda,\mu,1) [\operatorname{Spec}R^{\widetilde{\lambda},1}_{\overline{\rho}} \otimes_{\mathcal{O}} \mathbb{F}]
	$$
 	for $m(\lambda,\mu,1)$ the multiplicity of $\lambda$ in $V(\mu,1)$.
\end{theorem}
\begin{proof}
	First, by a simple twisting argument, we can assume $\mu_{\kappa,2} = 0$ for each $\kappa$. Since $\mu_{\kappa,1} \geq 1$ the bound $\sum_{\kappa|_{k} = \kappa_0}  \mu_{\kappa,1} \leq p$ implies that either $e<p$ or $e=p$ and $\mu_{\kappa,1} =1$ for each $\kappa$. In the latter case $\mu = \widetilde{\lambda}$ for $\lambda$ the trivial representation and in this case there is nothing to prove. Thus, we can assume $e<p$. This means $K$ is tamely ramified over $\mathbb{Q}_p$ and so Remark~\ref{tame} indicates that  Theorem~\ref{factor} applies. We can also assume that $e>1$, since if $e =1$ then again $\mu = \widetilde{\lambda}$ and the theorem is again trivial. This means that $\sum_{\kappa|_{k} = \kappa_0} \mu_{\kappa,1} <p+e-1$ for each $\kappa$ and so Proposition~\ref{isom} also applies, with $r_\kappa = \mu_{\kappa,1}$.
	
	Proposition~\ref{linearcomb} gives an identity of cycles
	$$
	[M_{\mu} \otimes_{\mathcal{O}} \mathbb{F}] = \sum_{\lambda} n(\lambda,\mu) [M_{\widetilde{\lambda}}\otimes_{\mathcal{O}} \mathbb{F}]
	$$
	for integers $n(\lambda,\mu) \geq 0$ and $\lambda$ running over tuples $(\lambda_{\kappa_0})$ with $\lambda_{\kappa_0} \leq \sum_{\kappa|_{k} = \kappa_0} (\mu_\kappa-\rho)$.  Each such $\lambda_{\kappa_0}$ then satisfies $\lambda_{\kappa_0,1} - \lambda_{\kappa_0,2} \leq p-1$ so we can also view the sum as running over absolutely irreducible $\mathbb{F}$-representations of $\operatorname{GL}_2(k)$ by \ref{absirred}. Applying the automorphism from Section~\ref{dual} gives
		$$
	[M_{-w_0\mu}\otimes_{\mathcal{O}} \mathbb{F}] = \sum_{\lambda} n(\lambda,\mu) [M_{-w_0\widetilde{\lambda}}\otimes_{\mathcal{O}} \mathbb{F}]
	$$
	Proposition~\ref{nablacontain} allows us to view this identity of cycles as occurring within the closed subscheme $\operatorname{Gr}^{\nabla_{\sigma},r}_{\mathcal{O}} \otimes_{\mathcal{O}} \mathbb{F}$ from \ref{Grnablasigmar} for $r= (r_\kappa)$. We want to consider its preimage under the composite
	\begin{equation}\label{compo}
	Y^{\leq h}_2 \times_{Z^{\leq h}_2} (\widetilde{Z}^{\nabla_{\sigma},r}_2 \otimes_{\mathcal{O}} \mathbb{F}) \rightarrow\widetilde{Z}^{\nabla_{\sigma},r}_2 \otimes_{\mathcal{O}} \mathbb{F} \rightarrow \operatorname{Gr}^{\nabla_{\sigma},r}_{\mathcal{O}} \otimes_{\mathcal{O}} \mathbb{F}	
\end{equation}
	(here the auxilliary integer $N$ is chosen sufficiently large so that Proposition~\ref{basiclocalmodel} applies). To do this we need to show the composite is flat. As the first map is an isomorphism by Proposition~\ref{isom}, and $\mathcal{G}_N$ is a smooth and irreducible group scheme, this composite is smooth with irreducible fibres. Smooth morphisms are flat so the pull-back of cycles is well defined and we obtain
	$$
	[Y_2^{\mu,\operatorname{flag}}] = \sum_{\lambda} n(\lambda,\mu) [Y^{\widetilde{\lambda},\operatorname{flag}}_2]
	$$
	where $Y_2^{\mu,\operatorname{flag}}$ denotes the preimage of $M_{-w_0\mu}\otimes_{\mathcal{O}} \mathbb{F}$. These are identities of $\operatorname{dim} \mathcal{G}_N + \sum_{\kappa} \operatorname{dim} G/P_{\mu_\kappa}$-dimensional cycles. Theorem~\ref{factor} shows that 
	$$
	\overline{Y}^\mu_2 \times_{Z^{\leq h}_2} \widetilde{Z}^{\leq h,N}_2 \hookrightarrow Y^{\mu,\operatorname{flag}}_2 ,\qquad 
	$$
	from which we conclude that $[\overline{Y}^\mu_2 \times_{Z^{\leq h}_2} \widetilde{Z}^{\leq h,N}_2] \leq [Y^{\mu,\operatorname{flag}}_2]$ as cycles. We also point out that since each $M_{\widetilde{\lambda}}\otimes_{\mathcal{O}} \mathbb{F}$ is irreducible and generically reduced (see Proposition~\ref{linearcomb}) the same is true of $M_{-w_0\widetilde{\lambda}}\otimes_{\mathcal{O}} \mathbb{F}$. The same is then also true of $Y^{\widetilde{\lambda},\operatorname{flag}}_2$ since \eqref{compo} is smooth with irreducible fibres. In particular, this implies the inequality $[\overline{Y}^\mu_2 \times_{Z^{\leq h}_2} \widetilde{Z}^{\leq h,N}_2] \leq [Y^{\mu,\operatorname{flag}}_2]$ is an equality when $\mu = \widetilde{\lambda}$. As a consequence
	$$
	[\overline{Y}^\mu_2 \times_{Z^{\leq h}_2} \widetilde{Z}^{\leq h,N}_2] \leq \sum_{\lambda} n(\lambda,\mu)[\overline{Y}^{\widetilde{\lambda}}_2 \times_{Z^{\leq h}_2} \widetilde{Z}^{\leq h,N}_2]
	$$
	as $\operatorname{dim} \mathcal{G}_N + \sum_{\kappa} \operatorname{dim} G/P_{\mu_\kappa}$-dimensional cycles inside the scheme $Y^{\leq h}_2 \times_{Z^{\leq h}_2} (\widetilde{Z}^{\nabla_{\sigma},r}_2 \otimes_{\mathcal{O}} \mathbb{F})$.
	
	The next goal is to descend this identity to an inequality of cycles in $\operatorname{Spec}R_{\overline{\rho}} \otimes_{\mathcal{O}} \mathbb{F}$. For this we recall the projective $R_{\overline{\rho}}$-scheme $\mathcal{L}^{\leq h}_{\overline{\rho}}$ introduced in the proof of Lemma~\ref{lift}. There is a formally smooth morphism $\mathcal{L}^{\leq h}_{\overline{\rho}} \rightarrow Y^{\leq h}_2$ with relative dimension $d^2 = 4$. Pulling back the previous inequality along the special fibre of $\mathcal{L}^{\leq h}_2 \times_{Z^{\leq h}_2} \widetilde{Z}^{\leq h,N}_2 \rightarrow Y^{\leq h}_2 \times_{Z^{\leq h}_2} \widetilde{Z}^{\leq h,N}_2$ (being a formally smooth morphism between Noetherian schemes, this map is flat and so the pull-back is defined) gives an inequality 
	$$
	[\overline{\mathcal{L}}^\mu_{\overline{\rho}} \times_{Z^{\leq h}_2} \widetilde{Z}^{\leq h,N}_2] \leq \sum_{\lambda} n(\lambda,\mu)[\overline{\mathcal{L}}^{\widetilde{\lambda}}_{\overline{\rho}} \times_{Z^{\leq h}_2} \widetilde{Z}^{\leq h,N}_2]
	$$
	where $\overline{\mathcal{L}}^{\mu}_{\overline{\rho}} = \mathcal{L}^{\mu}_{\overline{\rho}} \otimes_{\mathcal{O}} \mathbb{F}$ for $\mathcal{L}^{\mu}_{\overline{\rho}}$ the preimage of $Y^{\mu}_2$ in $\mathcal{L}^{\leq h}_{\overline{\rho}}$ (defined just as in the proof of Lemma~\ref{lift}). This is an identity of $d^2+\operatorname{dim} \mathcal{G}_N + \sum_{\kappa} \operatorname{dim} G/P_{\mu_\kappa}$-dimensional cycles. Since the morphism $\mathcal{L}^{\leq h}_{\overline{\rho}} \times_{Z^{\leq h}_2} (\widetilde{Z}^{\leq h,N}_2 \otimes_{\mathcal{O}} \mathbb{F}) \rightarrow \mathcal{L}^{\leq h}_{\overline{\rho}}$ is a $\mathcal{G}_N$-torsor (in particular smooth, surjective, and of relative dimension $\operatorname{dim}\mathcal{G}_N$) it follows that
	$$
	[\overline{\mathcal{L}}^\mu_{\overline{\rho}} ] \leq \sum_{\lambda} n(\lambda,\mu)[\overline{\mathcal{L}}^{\widetilde{\lambda}}_{\overline{\rho}} ]
	$$
	as $d^2 + \sum_\kappa \operatorname{dim} G/P_\kappa$-dimensional cycles inside $\mathcal{L}^{\leq h}_{\overline{\rho}}$. Recall that the projective morphism $\Theta:\mathcal{L}^{\leq h}_{\overline{\rho}} \rightarrow \operatorname{Spec} R_{\overline{\rho}}$ becomes a closed immersion after inverting $p$ and this closed immersion identifies $\mathcal{L}^{\mu}_{\overline{\rho}}[\frac{1}{p}] = \operatorname{Spec}R^{\mu,1}_{\overline{\rho}}[\frac{1}{p}]$. This was discussed in the proof of Lemma~\ref{lift}. Since the $R^{\mu}_{\overline{\rho}}$ are $\mathcal{O}$-flat an application of Lemma~\ref{spec} shows that
	$$
	\Theta_*[\overline{\mathcal{L}}^\mu_{\overline{\rho}}] = [\operatorname{Spec}R^{\mu,1}_{\overline{\rho}} \otimes_{\mathcal{O}} \mathbb{F}]
	$$
	Therefore, pushing forward the previous inequality of cycles gives
	$$
	[\operatorname{Spec}R^{\mu,1}_{\overline{\rho}} \otimes_{\mathcal{O}} \mathbb{F}] \leq \sum_{\lambda} n(\lambda,\mu)[\operatorname{Spec}R^{\widetilde{\lambda},1}_{\overline{\rho}} \otimes_{\mathcal{O}} \mathbb{F}]
	$$
	now as $d^2 + \sum_\kappa \operatorname{dim} G/P_\kappa$-dimensional cycles inside $\operatorname{Spec}R_{\overline{\rho}} \otimes_{\mathcal{O}} \mathbb{F}$. Proposition~\ref{global} then gives that $n(\lambda,\mu) \geq m(\lambda,\mu,1)$. Combining Lemma~\ref{compare} and Proposition~\ref{upperbound} shows that this must be an equality. The theorem follows.
\end{proof}
\section{Miscellany}

Let $\mathcal{X}$ and $\mathcal{Y}$ be algebraic stacks of finite type over a field $k$ and let $f:\mathcal{X}\rightarrow \mathcal{Y}$ be a morphism of stacks.
\begin{lemma}\label{mono}
	If, for $A$ any local finite $k$-algebra, the induced functor $\mathcal{X}(A) \rightarrow \mathcal{Y}(A)$:
	\begin{enumerate}
		\item is fully faithful then $f$ is a monomorphism (which by our definition implies being representable by algebraic spaces and separated).
		\item is an equivalence then $\mathcal{X}\rightarrow \mathcal{Y}$ is an isomorphism.
	\end{enumerate}
\end{lemma}
\begin{proof}
	First we prove (2) under the additional assumption that $f$ is representable by algebraic spaces. Then, by choosing a smooth surjection $U \rightarrow \mathcal{Y}$ with $U$ an algebraic space, we can assume that $\mathcal{X}$ and $\mathcal{Y}$ are algebraic spaces. With this reduction the argument given in \cite[7.2.4]{LLBBM20} goes through with schemes replaced by algebraic spaces. Indeed, by \cite[0APP]{stacks-project} this morphism is smooth and quasi-finite, and hence \'etale. By \cite[05W1]{stacks-project} the diagonal $\mathcal{X}\rightarrow \mathcal{X}\times_{\mathcal{Y}} \mathcal{X}$ is an open immersion. Since it is surjective on finite type points it is an isomorphism and so $\mathcal{X}\rightarrow \mathcal{Y}$ is a monomorphism. By \cite[05W5]{stacks-project} it is an open immersion, and so an isomorphism, again by surjectivity on finite type points.
	
	Now we prove (1). By \cite[04XS]{stacks-project} the diagonal $\Delta_f:\mathcal{X}\rightarrow \mathcal{X}\times_{\mathcal{Y}}\mathcal{X}$ is representable by algebraic spaces. Full faithfullness of $f$ on $A$-valued points implies that $\Delta_f$ is an equivalence on such points. Therefore the first paragraph implies $\Delta_f$ is an isomorphism. From \cite[04ZZ]{stacks-project} we obtain (1).
	
	To finish the proof of (2) note that by (1) we have $f$ representable by algebraic spaces.
\end{proof}
\begin{corollary}\label{factorising}
	Suppose $\mathcal{Z}$ is a closed substack of $\mathcal{Y}$ and that for every morphism $\operatorname{Spec}A\rightarrow \mathcal{X}$, with $A$ any local finite $k$-algebra, the composite $\operatorname{Spec}A \rightarrow \mathcal{X} \rightarrow \mathcal{Y}$ factors through $\mathcal{Z}$. Then $\mathcal{X}\rightarrow \mathcal{Y}$ factors through $\mathcal{Z}$.
\end{corollary}
\begin{proof}
	This follow since by Lemma~\ref{mono} the map $\mathcal{X} \rightarrow \mathcal{X} \times_{\mathcal{Y}} \mathcal{Z}$ is an isomorphism.
\end{proof}
\bibliography{/home/user/Dropbox/Maths/biblio.bib}
\end{document}